\pdfoutput=1

\documentclass[final]{siamltex}
\usepackage{graphicx}          		% standard LaTeX graphics tool when including figure files
\usepackage{graphics}         	 	% for pdf, bitmapped graphics files
\usepackage{epsfig}            		% for postscript graphics files
\usepackage{amsmath}           		% mathematics package
\usepackage{amssymb}           	% mathematic symbols
\usepackage{subfigure}        		% for placing multiple figures next to each other
\usepackage[usenames,dvipsnames]{color}   % defining colors for pdf tags with dvi2pdf
\newtheorem{example}[theorem]{Example}
\newtheorem{remark}[theorem]{remark}

\newcommand*\fvec[1]{\ensuremath{\mathbf{#1}}}                                	% creates a bold symbol
\newcommand*\dt[0]{\frac{d}{d\,t}\,}					                      	% time derivative d/dt
\newcommand{\abs}[1]{\left\lvert{#1}\right\rvert}                     			% \abs{x} -->  |x| with proper spacing
\newcommand{\norm}[1]{\left\lVert{#1}\right\rVert}                     		% \norm{x} -->  ||x|| with proper spacing
\newcommand*\mc[0]{\mathcal}                                                  		% abbreviation for \mathcal
\newcommand*\mbb[0]{\mathbb}                                                  		% abbreviation for \mathbb
\DeclareMathOperator*{\sinc}{sinc} 	                                       		% define the sinc operator
\newcommand{\until}[1]{\{1,\dots, #1\}}
\newcommand{\subscr}[2]{#1_{\textup{#2}}}

\newcommand{\setdef}[2]{\{#1 \; | \; #2\}}

\newcommand{\map}[3]{#1: #2 \rightarrow #3}

\newcommand{\real}{\mathbb{R}}

\newcommand\oprocendsymbol{\hbox{$\square$}}
\newcommand\oprocend{\relax\ifmmode\else\unskip\hfill\fi\oprocendsymbol}

\title{On the Critical Coupling for Kuramoto Oscillators\thanks{This work
    was supported in part by NSF grants IIS-0904501 and CNS-0834446.}  }

\author{Florian D\"orfler \and Francesco Bullo%
  \thanks{Florian D\"orfler and Francesco Bullo are with the Center for
    Control, Dynamical Systems and Computation, University of California at
    Santa Barbara, Santa Barbara, CA 93106, {\tt\small \{dorfler,
      bullo\}@engineering.ucsb.edu}} }

\begin{document}

\maketitle

%%%%%%%%%%%%%%%%%%%%%%%%%%%%%%%%%%%%%%%%%%%%%%%%%%%%%%%%%%%%%%%%%%%%%%%%%

% -----------------------------------------------------------------------------------------------------%
%      Abstract
% -----------------------------------------------------------------------------------------------------%

\begin{abstract}
The celebrated Kuramoto model captures various synchronization phenomena in biological and man-made dynamical systems of coupled oscillators. It is well-known that there exists a critical coupling strength among the oscillators at which a phase transition from incoherency to synchronization occurs. This paper features four contributions. First, we characterize and distinguish the different notions of synchronization used throughout the literature and formally introduce the concept of phase cohesiveness as an analysis tool and performance index for synchronization. Second, we review the vast literature providing necessary, sufficient, implicit, and explicit estimates of the critical coupling strength in the finite and infinite-dimensional case and for both first-order and second-order Kuramoto models. Third, we present the first explicit necessary and sufficient condition on the critical coupling strength to achieve synchronization in the finite-dimensional Kuramoto model for an arbitrary distribution of the natural frequencies. The multiplicative gap in the synchronization condition yields a practical stability result determining the admissible initial and the guaranteed ultimate phase cohesiveness as well as the guaranteed asymptotic magnitude of the order parameter. As supplementary results, we provide a statistical comparison of our synchronization condition with other conditions proposed in the literature, and we show that our results also hold for switching and smoothly time-varying natural frequencies. Fourth and finally, we extend our analysis to multi-rate Kuramoto models consisting of second-order Kuramoto oscillators with inertia and viscous damping together with first-order Kuramoto oscillators with multiple time constants. We prove that such a heterogenous network is locally topologically conjugate to a first-order Kuramoto model with scaled natural frequencies. Finally, we present necessary and sufficient conditions for almost global  phase synchronization and local frequency synchronization in the multi-rate Kuramoto model. Interestingly, our provably correct synchronization conditions do not depend on the inertiae which contradicts prior observations on the role of inertial effects in synchronization of second-order Kuramoto oscillators.
\end{abstract}

%%%%%%%%%%%%%%%%%%%%%%%%%%%%%%%%%%%%%%%%%%%%%%%%%%%%%%%%%%%%%%%%%%%%%%%%%


%\begin{AMS}
%...
%\end{AMS}

\pagestyle{myheadings}
\thispagestyle{plain}
\markboth{F. D\"orfler and F. Bullo}{On the Critical Coupling for Kuramoto Oscillators}

%%%%%%%%%%%%%%%%%%%%%%%%%%%%%%%%%%%%%%%%%%%%%%%%%%%%%%%%%%%%%%%%%%%%%%%%%

% -----------------------------------------------------------------------------------------------------%
%      Main Body
% -----------------------------------------------------------------------------------------------------%

% -----------------------------------------------------------------------------------------------------%
\section{Introduction}

A classic and celebrated model for the synchronization of coupled oscillators is due to Yoshiki Kuramoto \cite{YK:75}. The {\it Kuramoto model} considers $n \geq 2$ coupled oscillators each represented by a phase variable $\theta_{i} \in \mbb T^{1}$, the 1-tours, and a natural frequency $\omega_{i} \in \mbb R$. The system of coupled oscillators obeys the dynamics
\begin{equation}
	\dot \theta_{i}
	=
	\omega_{i} - \frac{K}{n} \sum\limits_{j=1}^{n} \sin(\theta_{i} - \theta_{j})
	\,,
	\quad i \in \until n
	\label{eq: Kuramoto model}
	\,,
\end{equation}
where  $K  > 0$ is the coupling strength among the oscillators.
 
The Kuramoto model \eqref{eq: Kuramoto model} finds application in various biological synchronization phenomena, and we refer the reader to the excellent reviews \cite{SHS:00,JAA-LLB-CJPV-FR-RS:05} for various references. Recent technological applications of the Kuramoto model include motion coordination of particles \cite{RS-DP-NEL:07}, synchronization in coupled Josephson junctions \cite{KW-PC-SHS:98}, transient stability analysis of power networks \cite{FD-FB:09o-arxiv}, and deep brain stimulation \cite{PAT:03}.

%A prototypical and celebrated model for the synchronization of coupled oscillators is due to Yoshiki Kuramoto \cite{YK:75}. 
%%
%The Kuramoto model finds application in various biological synchronization phenomena such as pacemaker cells in the heart \cite{DCM-EPM-JJ:87} and in the brain \cite{CL-DRW-SHS-RSM:97}, metabolic synchrony in yeast cell populations \cite{AKG-BC-EKP:71}, flashing fireflies \cite{JB:88}, and chirping crickets \cite{TJW:69}.
%%
%Recent technological applications of the Kuramoto model include motion coordination of particles \cite{RS-DP-NEL:07}, synchronization in coupled Josephson junctions \cite{KW-PC-SHS:98}, deep brain stimulation \cite{PAT:03}, semiconductor lasers \cite{GK-AGV-PM:00}, microwave oscillators \cite{RAY-RCC:02}, and clock synchronization in wireless networks \cite{OS-US-YBN-SS:08}.
%%
%Excellent reviews concerning theoretic developments and applications of Kuramoto oscillators can be found in \cite{SHS:00,JAA-LLB-CJPV-FR-RS:05}. 
%%
%The power network dynamics are strikingly similar to the Kuramoto model. Due to the striking similarity between power network models and the Kuramoto model, a Kuramoto-type analysis for transient stability problems has been proposed by the power networks community \cite{GF-AHN-NFP:08,VF-SR-EC-EM-VR:09,DS-UR-BS-MP:01}, by the dynamical systems community \cite{AA-ADG-JK-YM-CZ:08,HAT-AJL-SO:97}, and by ourselves \cite{FD-FB:09o-arxiv}[new paper]

% -----------------------------------------------------------------------------------------------------%
\subsection*{The Critical Coupling Strength}
Yoshiki Kuramoto himself analyzed the model \eqref{eq: Kuramoto model} based on the
order parameter $r e^{\mathrm{i} \psi} = \frac{1}{n} \sum_{j=1}^{n}
e^{\mathrm{i} \theta_{j}}$, which corresponds the centroid of all
oscillators when represented as points on the unit circle in $\mbb
C^{1}$. The magnitude of the order parameter can be understood as a measure
of synchronization. If the angles $\theta_{i}(t)$ of all oscillators are identical, then $r=1$, and if all
oscillators are spaced equally on the unit circle (splay state), then
$r=0$. With the help of the order parameter, the Kuramoto model \eqref{eq:
  Kuramoto model} can be written in the insightful form
\begin{equation}
	\dot \theta_{i} = \omega_{i} - K r \sin(\theta_{i} - \psi)
	\,,
	\quad i \in \until n
	\label{eq: Kuramoto model with order parameter}
	\,.
\end{equation}
Equation \eqref{eq: Kuramoto model with order parameter} gives the
intuition that the oscillators synchronize by coupling to a mean field
represented by the order parameter $r e^{\mathrm{i} \psi}$. Intuitively,
for small coupling strength $K$ each oscillator rotates with its natural
frequency $\omega_{i}$, whereas for large coupling strength $K$ all angles
$\theta_{i}(t)$ will be entrained by the mean field $r e^{\mathrm{i} \psi}$
and the oscillators synchronize. The threshold from incoherency to
synchronization occurs for some critical coupling
$\subscr{K}{critical}$. This phase transition has been the source of
numerous research papers starting with Kuramoto's own insightful and
ingenuous analysis \cite{YK:75,YK:84}. For instance, since $r \leq 1$, no
solution of \eqref{eq: Kuramoto model with order parameter} of the form
$\dot \theta_{i}(t) = \dot \theta_{j}(t)$ can exist if $K < |\omega_{i} -
\omega_{j}|/2$. Hence, $K \geq |\omega_{i} - \omega_{j}|/2$ provides a
necessary synchronization condition and a lower bound for
$\subscr{K}{critical}$. Various various necessary, sufficient, implicit, and explicit estimates of the critical coupling strength
$\subscr{K}{critical}$ for both the on-set as well as the ultimate stage of
synchronization have been derived in the vast literature on the Kuramoto model \cite{YK:75,SHS:00,JAA-LLB-CJPV-FR-RS:05,RS-DP-NEL:07,FD-FB:09o-arxiv,YK:84,GBE:85,EAM-EB-SHS-PS-TMA:09,JLvH-WFW:93,NC-MWS:08,AJ-NM-MB:04,SJC-JJS:10,GSS-UM-FA:09,FDS-DA:07,MV-OM:08,REM-SHS:05,DA-JAR:04,ZL-BF-MM:07,EC-PM:08,MV-OM:09,LB-LS-ADG:09,AF-AC-WPL:10,SYH-TH-KJH:10}. To date	, only explicit and sufficient (or necessary) bounds are known for the critical coupling strength $\subscr{K}{critical}$ in the Kuramoto model \eqref{eq: Kuramoto model}, and implicit formulae are available to compute the exact value of $\subscr{K}{critical}$.

% -----------------------------------------------------------------------------------------------------%
\subsection*{The Multi-Rate Kuramoto Model}
As second relevant coupled-oscillator model, consider $m \geq 0$  second-order Kuramoto oscillators with inertia and viscous damping and $n - m \geq 0$ first-order Kuramoto oscillators with multiple time constants. The {\em multi-rate Kuramoto model} evolving on $\mbb T^{n} \times \mbb R^{m}$ then reads as
\begin{align}
	\begin{split}
	M_{i} \ddot \theta_{i} + D_{i} \dot \theta_{i}
	&=
	\omega_{i} - \frac{K}{n} \sum_{i=1}^{n} \sin(\theta_{i}-\theta_{j})
	\,, \quad i \in \until m\,,
	\\
	D_{i} \dot \theta_{i}
	&=
	\omega_{i} - \frac{K}{n} \sum_{i=1}^{n} \sin(\theta_{i}-\theta_{j})
	\,, \quad i \in \{m+1,\dots,n\}\,,
	\end{split}
	\label{eq: multi-rate Kuramoto model}
\end{align}
where $M_{i}>0$, $D_{i}>0$, and $\omega_{i} \in \mathbb R$ for $i \in \until n$ and $K>0$. Note that\,\,we\,\,allow for $m \in \{0,n\}$ such that the model \eqref{eq: multi-rate Kuramoto model} is of purely first or second order, respectively.

The multi-rate Kuramoto model \eqref{eq: multi-rate Kuramoto model} finds explicit application in the classic structure-preserving power network model proposed in \cite{ARB-DJH:81}. For $m=n$, the model \eqref{eq: multi-rate Kuramoto model} is a purely second-order system of coupled, damped, and driven pendula, which has been used, for example, to model synchronization in a population of fireflies \cite{GBE:91}, in coupled Josephson junctions \cite{KW-PC-SHS:98}, and in network-reduced power system models\,\cite{HDC-CCC:95}.

For $m=n$, unit damping $D_{i}=1$, and uniform inertia $M_{i} = M > 0$, the second-order Kuramoto system \eqref{eq: multi-rate Kuramoto model} has been extensively studied in the literature \cite{YPC-SYH-SBH:11,HAT-AJL-SO:97,HAT-AJL-SO:97b,HH-GSJ-MYC:02,HH-MYC-JY-KSS:99,JAA-LLB-RS:00,JAA-LLB-CJPV-FR-RS:05}. The cited results on the inertial effects on synchronization are controversial and report that synchronization is either enhanced or inhibited by sufficiently large (or also sufficiently small) inertia $M$. For the general multi-rate Kuramoto model \eqref{eq: multi-rate Kuramoto model} no exact synchronization conditions are known.

% -----------------------------------------------------------------------------------------------------%
\subsection{Contributions}

The contributions of this paper are four-fold. 
First, we characterize, distinguish, and relate different concepts of synchronization and their analysis methods, which are studied and employed in the networked control, physics, and dynamical systems communities. In particular, we review the concepts of phase synchronization and frequency synchronization, and introduce the notion of phase cohesiveness. In essence, a solution to the Kuramoto model \eqref{eq: Kuramoto model} is phase cohesive if all angles are bounded within a (possibly rotating) arc of fixed length. The notion of phase cohesiveness provides a powerful analysis tool for synchronization and can be understood as a performance index for synchronization similar to the order parameter. 

As second contribution, we review the extensive literature on the Kuramoto model, and present various necessary, sufficient, implicit, and explicit estimates of the critical coupling strength for the finite and infinite-dimensional Kuramoto model in a unified language \cite{YK:75,SHS:00,JAA-LLB-CJPV-FR-RS:05,RS-DP-NEL:07,FD-FB:09o-arxiv,YK:84,GBE:85,EAM-EB-SHS-PS-TMA:09,JLvH-WFW:93,NC-MWS:08,AJ-NM-MB:04,SJC-JJS:10,GSS-UM-FA:09,FDS-DA:07,MV-OM:08,REM-SHS:05,DA-JAR:04,ZL-BF-MM:07,EC-PM:08,MV-OM:09,LB-LS-ADG:09,AF-AC-WPL:10,SYH-TH-KJH:10}. Aside from the comparison of the different estimates of the critical coupling strength, the second purpose of this review is the comparison of the different analysis techniques. Furthermore, we briefly survey the controversial results  \cite{YPC-SYH-SBH:11,HAT-AJL-SO:97,HAT-AJL-SO:97b,HH-GSJ-MYC:02,HH-MYC-JY-KSS:99,JAA-LLB-RS:00,JAA-LLB-CJPV-FR-RS:05} on the role of inertia in second-order Kuramoto models.

As third contribution of this paper, we provide an
explicit necessary and sufficient condition on the critical coupling
strength to achieve exponential synchronization in the finite-dimensional
Kuramoto model for an arbitrary distribution of the natural frequencies
$\omega_{i}$, see Theorem \ref{Theorem: Necessary and Sufficient conditions}. In particular, synchronization occurs for $K >
\subscr{K}{critical} = \subscr{\omega}{max} - \subscr{\omega}{min}$, where
$\subscr{\omega}{max}$ and $\subscr{\omega}{min}$ are the maximum and
minimum natural frequency, respectively. The multiplicative gap $\subscr{K}{critical}/K$ determines the admissible initial and the guaranteed
ultimate level of phase cohesiveness as well as the guaranteed asymptotic
magnitude $r$ of the order parameter. In particular, the ultimate level of
phase cohesiveness can be made arbitrary small by increasing the
multiplicative gap $\subscr{K}{critical}/K$.  This result resembles the
concept of {\it practical stability} in dynamics and control if $K$ and
$\subscr{K}{critical}$ are understood as a synchronization-enhancing gain
and as a measure for the desynchronizing non-uniformity among the
oscillators. Additionally, our main result includes estimates on the
exponential synchronization rate for phase and frequency synchronization. We further provide two supplementary results on our synchronization condition. In statistical studies, we compare our condition to other necessary and explicit or implicit and exact conditions proposed in the literature. Finally, we show that our analysis and the resulting synchronization conditions also hold for  switching and smoothly time-varying natural frequencies.

%Compared to the authors' earlier work in power networks \cite{FD-FB:09o-arxiv}, the third contribution proves necessity of the sufficient bound on the critical coupling derived in \cite{FD-FB:09o-arxiv} and doubles the corresponding estimate for the region of attraction given in \cite{FD-FB:09o-arxiv}.

As fourth and final contribution, we extend our main result on the classic Kuramoto model \eqref{eq: Kuramoto model} to the multi-rate Kuramoto model \eqref{eq: multi-rate Kuramoto model}. We prove a general result that relates the equilibria and local stability properties of forced gradient-like systems to those of dissipative Hamiltonian systems together with gradient-like dynamics and external forcing, see Theorem \ref{Theorem: Properties of Hlambda family}. As special case, we are able to show that the multi-rate Kuramoto model is locally topologically conjugate to a first-order Kuramoto model with scaled natural frequencies, see Theorem \ref{Theorem: Local Equivalence of First and Second Order Synchronization}. Finally, we present necessary and sufficient conditions for almost global stability of phase synchronization and local stability of frequency synchronization in the multi-rate Kuramoto model, see Theorem \ref{Theorem: Exponential Synchronization in the Multi-Rate Kuramoto Model}. Interestingly, the inertial coefficients $M_{i}$ {\em do not affect} the synchronization conditions and the asymptotic synchronization frequency. Moreover, the location and local stability properties of all equilibria are independent of the inertial coefficients $M_{i}$, and so are all local bifurcations and the the asymptotic magnitude of the order parameter. Rather, these quantities depend on the viscous damping parameters $D_{i}$ and the natural frequencies $\omega_{i}$. Of course, the inertial terms still affect the transient synchronization behavior which lies outside the scope of our analysis. These interesting and provably correct findings contradict prior observations on the role of inertia inhibiting or enhancing synchronization in second-order Kuramoto models.

The remainder of this paper is organized as follows. Section \ref{Section: Definition of Sync} reviews different concepts of synchronization and provides a motivating example. Section \ref{Section: Review of Critical Bounds} reviews the literature on the critical coupling strength in the Kuramoto model. Section \ref{Section: New, tight, and explicit bound} presents a novel, tight, and explicit bound on the critical coupling as well as various related properties, performance estimates, statistical studies, and extensions to time-varying natural frequencies. Section \ref{Section: Synchronization of Multi-Rate Kuramoto Models} extends some of these results to the multi-rate Kuramoto model. Finally, Section \ref{Section: Conclusions} concludes the paper.

\subsubsection*{Notation} The {\it torus} is the set $\mbb T^{1} =
{]\!-\!\pi,+\pi]}$, where $-\pi$ and $+\pi$ are associated with each other,
an {\it angle} is a point $\theta \in \mbb T^{1}$, and an {\it arc} is a
connected subset of $\mbb T^{1}$. The product set $\mbb T^{n}$ is the
$n$-dimensional torus.  With slight abuse of notation, let
$|\theta_{1}-\theta_{2}|$ denote the {\it geodesic distance} between two
angles $\theta_{1} \in \mbb T^{1}$ and $\theta_{2} \in \mbb T^{1}$. For
$\gamma\in{[0,\pi]}$, let $\Delta(\gamma) \subset \mbb T^{n}$ be the set of
angle arrays $(\theta_1,\dots,\theta_n)$ with the property that there
exists an arc of length $\gamma$ containing all $\theta_1,\dots,\theta_n$
in its interior. Thus, an angle array $\theta \in \Delta(\gamma)$
satisfies $\max\nolimits_{i,j \in \until n} |\theta_{i} - \theta_{j}| <
\gamma$.  For $\gamma\in{[0,\pi]}$, we also define $\bar\Delta(\gamma)$ to
be the union of the phase-synchronized set $\{\theta \in \mbb T^{n} \;|\;
\theta_{i} = \theta_{j} ,\, i,j \in \until n\}$ and the closure of the open
set $\Delta(\gamma)$. Hence, $\theta \in \bar \Delta(\gamma)$ satisfies
$\max\nolimits_{i,j \in \until n} |\theta_{i} - \theta_{j}| \leq \gamma$;
the case $\theta \in \bar \Delta(0)$ corresponds simply to $\theta$ taking
value in the phase-synchronized set.  

 Given an $n$-tuple $(x_{1},\dots,x_{n})$, let $x \in \mbb R^{n}$ be the associated vector, let $\diag(x_{i}) \in \mbb R^{n}$ be the associated diagonal matrix, and let
$\subscr{x}{max}$ and $\subscr{x}{min}$ be the maximum and minimum elements. The {\it inertia} of a matrix $A \in \mbb R^{n \times n}$ are given by the triple $\{\subscr{\nu}{s},\subscr{\nu}{c},\subscr{\nu}{u}\}$, where $\subscr{\nu}{s}$ (respectively $\subscr{\nu}{u}$) denotes the number of stable (respectively unstable) eigenvalues of $A$ in the open left (respectively right) complex half plane, and $\subscr{\nu}{c}$ denotes the number of center eigenvalues with zero real part. The notation $\mathrm{blkdiag}(A_{1},\dots,A_{n})$ denotes the block-diagonal matrix with matrix blocks $A_{1},\dots,A_{n}$. Finally, let $I_{n}$ be the $n$-dimensional identity matrix, and let $\fvec 1_{p \times q}$ and $\fvec 0_{p \times q}$ denote the $p \times q$ dimensional matrix with unit entries and zero entries, respectively.

% -----------------------------------------------------------------------------------------------------%
\section{Phase Synchronization, Phase Cohesiveness, and Frequency Entrainment}
\label{Section: Definition of Sync}
% -----------------------------------------------------------------------------------------------------%

Different levels of synchronization are typically distinguished for the Kuramoto model \eqref{eq: Kuramoto model}. 
The case when all angles $\theta_{i}(t)$ converge exponentially to a common angle $\theta_{\infty} \in \mbb T^{1}$ as $t \to \infty$ is referred to as {\it exponential phase synchronization} and can only occur if all natural frequencies are identical. 
%In this case the order parameter takes the value $r=1$. 
If the natural frequencies are non-identical, then each pairwise distance $|\theta_{i}(t) - \theta_{j}(t)|$ can converge to a constant value, but this value is
not necessarily zero. The following concept of phase cohesiveness addresses exactly this point.
A solution $\map{\theta}{\real_{\geq0}}{\mbb T^n}$ to the Kuramoto
model~\eqref{eq: Kuramoto model} is \emph{phase cohesive} if there exists
a length $\gamma \in [0,\pi[$ such that $\theta(t)\in\bar\Delta(\gamma)$ for all
$t\geq0$, i.e., at each time $t$ there exists an arc of length $\gamma$
containing all angles $\theta_i(t)$.  A solution
$\map{\theta}{\real_{\geq0}}{\mbb T^n}$ achieves \emph{exponential frequency synchronization} if all frequencies $\dot \theta_{i}(t)$ converge exponentially fast to a common frequency $\dot \theta_{\infty} \in \mbb R$ as $t \to \infty$.  Finally, a solution
$\map{\theta}{\real_{\geq0}}{\mbb T^n}$ achieves {\it exponential
  synchronization} if it is phase cohesive and it achieves exponential
frequency synchronization.

If a solution $\theta(t)$ achieves exponential frequency synchronization, all phases asymptotically become constant in a rotating coordinate frame with frequency $\dot \theta_{\infty}$, or equivalently, all phase distances $|\theta_{i}(t) - \theta_{j}(t)|$ asymptotically become constant. Hence, the terminology {\it phase locking} is sometimes also used in the literature to define a solution $\map{\theta}{\real_{\geq0}}{\mbb T^n}$ that satisfies $\dot \theta_{i}(t) = \dot \theta_{\infty}$ for all $i \in \until n$ and for all $t \geq 0$ \cite{REM-SHS:05,JLvH-WFW:93,AF-AC-WPL:10} or $\theta_{i}(t) - \theta_{j}(t) = constant$ for all $i,j \in \until n$ and for all $t \geq 0$ \cite{DA-JAR:04,GBE:85,EC-PM:08,MV-OM:08,MV-OM:09}. 
Other commonly used terms in the vast synchronization literature include full, exact, or perfect synchronization for phase synchronization\footnote{Note that \cite{JAA-LLB-CJPV-FR-RS:05} understands phase locking synonymous to phase synchronization as defined above.} and\,\,frequency locking, frequency entrainment, or partial synchronization for frequency synchronization.

In the networked control community, boundedness of angular distances and
consensus arguments are typically combined to establish frequency
synchronization
\cite{NC-MWS:08,AJ-NM-MB:04,GSS-UM-FA:09,FD-FB:09o-arxiv,AF-AC-WPL:10,SYH-TH-KJH:10}.
Our latter analysis in Section \ref{Section: New, tight, and explicit bound}
makes this approach explicit by distinguishing between
phase cohesiveness and frequency synchronization. Note that phase cohesiveness
can also be understood as a performance measure for synchronization and
phase synchronization is simply the extreme case of phase cohesiveness with
$\lim_{t \to \infty} \theta(t) \in \bar\Delta(0)$.
Indeed, if the magnitude $r$ of the order
parameter is understood as an {\it average} performance index for
synchronization, then phase cohesiveness can be understood as a {\it
  worst-case} performance index. The following lemma relates the magnitude
of the order parameter to a guaranteed level of phase cohesiveness.

\begin{lemma}[\bf Phase cohesiveness and order parameter]
\label{Lemma: phase cohesiveness and order parameter}
  Consider an array of $n \geq 2$ angles $\theta = (\theta_{1}, \dots,
  \theta_{n}) \in \mbb T^n$ and compute the magnitude of the order
  parameter $r(\theta) = \frac{1}{n} |\sum_{j=1}^{n}
  e^{\mathrm{i}\theta_{j}} |$.  The following statements hold:
  \begin{enumerate}
  \item[1)] if $\theta \in \bar\Delta(\gamma)$ for some $\gamma \in
    {[0,\pi]}$, then $r(\theta) \in [\cos(\gamma/2), 1]$; and

  \item[2)] if $r(\theta) \in {[0,1]}$ and $\theta \in\bar\Delta(\pi)$,
    then $\theta \in \bar\Delta(\gamma)$ for some $\gamma \in
    [2\arccos(r(\theta)), \pi]$.
\end{enumerate}
%If $\theta \in \bar\Delta(\gamma)$ for some $\gamma \in {[0,\pi[}$, then it holds that $\cos(\gamma/2) \leq r \leq 1$. Conversely, if $r \in {[0,1]}$, then $\theta \in \bar\Delta(\gamma)$, where $\pi \leq \gamma \leq 2\arccos(r)$.
\end{lemma}

\begin{proof}
  As customary, we abbreviate $r(\theta)$ with $r$ in what follows.  The
  order parameter $r e^{\mathrm{i}\psi}$ is the centroid of all {\it
    phasors} $e^{\mathrm{i} \theta_{j}}$ corresponding to the phases
  $\theta_{j}$ when represented as points on the unit circle in $\mbb
  C^{1}$. Hence, for $\theta \in \bar\Delta(\gamma)$, $\gamma \in
  {[0,\pi]}$, it follows that $r$ is contained in the convex hull of the
  arc of length $\gamma$, as illustrated in Figure \ref{Fig: arc and its
    convex hull}.
  \begin{figure}[htbp]
    \centering{
      \includegraphics[scale=0.45]{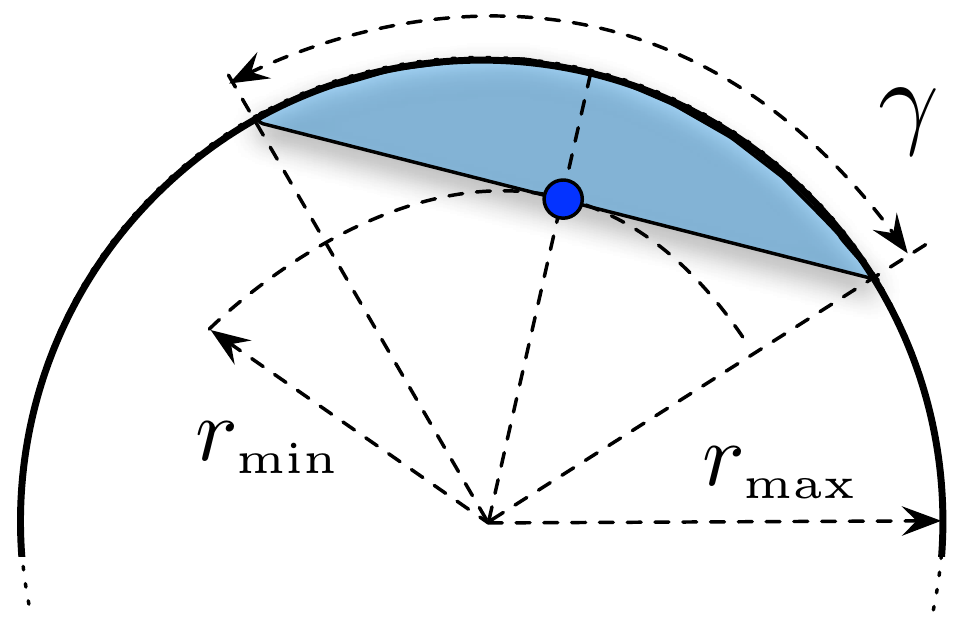}
      \caption{Schematic illustration of an arc of length $\gamma \in {[0,\pi]}$, its
        convex hull (shaded), and the location {\large\bf${\color{blue}\fvec\bullet}$} of the corresponding order parameter $r e^{\mathrm{i}\psi}$ with minimum magnitude $\subscr{r}{min}$.}
      \label{Fig: arc and its convex hull}
    }
  \end{figure}
  Let $\gamma \in {[0,\pi]}$ be fixed and let $\theta \in
  \bar\Delta(\gamma)$. It follows from elementary geometric arguments that
  $\cos(\gamma/2) = \subscr{r}{min} \leq r \leq \subscr{r}{max} = 1$, which
  proves statement 1). Conversely, if $r$ is fixed and $\theta \in
  \bar\Delta(\pi)$, then the centroid $r e^{\mathrm{i}\psi}$ is always
  contained within the convex hull of the semi-circle $\bar\Delta(\pi)$
  (centered at $\psi$). The smallest arc whose convex hull contains the
  centroid $r e^{\mathrm{i}\psi}$ is the arc of length\,\,$\gamma =
  2\arccos(r)$ (centered at $\psi$), as illustrated in Figure \ref{Fig: arc
    and its convex hull}. This proves statement 2).
\end{proof}

In the physics and dynamical systems community exponential synchronization is usually analyzed in relative coordinates. For instance, since the average frequency $\frac{1}{n} \sum_{i=1}^{n} \dot \theta_{i}(t) = \frac{1}{n} \sum_{i=1}^{n} \omega_{i} \triangleq \subscr{\omega}{avg}$ is constant, the Kuramoto model \eqref{eq: Kuramoto model} is sometimes \cite{MV-OM:08,REM-SHS:05} analyzed with respect to a rotating frame in the coordinates $\xi_{i} = \theta_{i} - \subscr{\omega}{avg} t \pmod{2\pi}$, $i \in \until n$, corresponding to a deviation from the average angle. The existence of an exponentially stable one-dimensional (due to translational invariance) equilibrium manifold in $\xi$-coordinates then implies local stability of phase-locked solutions and exponential synchronization. Alternatively, the translational invariance can be removed by formulating the Kuramoto model \eqref{eq: Kuramoto model} in {\it grounded coordinates}  $\delta_i=\theta_i-\theta_n$, for $i\in\until{n-1}$ \cite{FD-FB:09o-arxiv,DA-JAR:04}. We refer to \cite[Lemma IV.1]{FD-FB:09o-arxiv} for a geometrically rigorous characterization of the grounded $\delta$-coordinates and the relation of exponential stability in $\delta$-coordinates and exponential synchronization in $\theta$-coordinates.

The following example of two oscillators illustrates the notion of phase cohesiveness, applies graphical synchronization analysis techniques, and points out various important geometric subtleties occurring on the compact state space\,$\mbb T^{2}$.

\begin{example}[\bf Two oscillators]
\label{Example: Two coupled Kuramoto Oscillators}
\normalfont Consider $n=2$ oscillators with $\omega_{2}>\omega_{1}$. We restrict our attention to angles contained in an open
half-circle: for angles $\theta_{1}$, $\theta_{2}$ with $|\theta_{2} -
\theta_{1}|<\pi$, we define the {\it angular difference} $\theta_{2}-\theta_{1}$ to be the number in
${]\!-\!\pi,\pi[}$ with magnitude equal to the geodesic distance
$|\theta_{2} - \theta_{1}|$ and with positive sign iff the
counter-clockwise path length from $\theta_{1}$ to $\theta_{2}$ on $\mbb
T^{1}$ is smaller than the clockwise path length. With this definition the two-dimensional Kuramoto dynamics $(\dot \theta_{1},\dot \theta_{2})$ can be reduced to the 
scalar difference dynamics\,\,$\dot \theta_{2} - \dot \theta_{1}$. After scaling time as $t \mapsto t (\omega_{2} - \omega_{1})$
and introducing $\kappa = K/(\omega_{2} - \omega_{1})$ the difference dynamics are
\begin{equation}
	\dt (\theta_{2} - \theta_{1} )
	=
	f_{\kappa}(\theta_{2} - \theta_{1})
	:=
	1 - \kappa \sin(\theta_{2} - \theta_{1})
	\label{eq: difference dynamics for two oscillators}
	\,.
\end{equation}
The scalar dynamics \eqref{eq: difference dynamics for two oscillators} can be analyzed graphically by plotting the vector field $f_{\kappa}(\theta_{2} - \theta_{1})$ over the difference variable $\theta_{2} - \theta_{1}$, as in Figure \ref{Fig: vector field for two oscillators}. 
Figure \ref{Fig: vector field for two oscillators} displays a saddle-node bifurcation at $\kappa=1$. For $\kappa < 1$ no equilibrium of \eqref{eq: difference dynamics for two oscillators} exists, and for $\kappa >1$ an asymptotically stable equilibrium $\subscr{\theta}{stable} = \arcsin(\kappa^{-1}) \in {]0,\pi/2[}$ together with a saddle point $\subscr{\theta}{saddle} = \arcsin(\kappa^{-1}) \in {]\pi/2,\pi[}$ exists. 
%These equilibria satisfy $\sin(\subscr{\theta}{stable}) = \sin(\subscr{\theta}{saddle}) = \kappa^{-1}$. 
%
For $\theta(0) \in \Delta(|\subscr{\theta}{saddle}|)$ all trajectories converge exponentially to $\subscr{\theta}{stable}$, that is, the oscillators synchronize exponentially. Additionally, the oscillators are phase cohesive iff $\theta(0) \in \bar\Delta(|\subscr{\theta}{saddle}|)$, where all trajectories remain bounded. For $\theta(0) \not \in \bar\Delta(|\subscr{\theta}{saddle}|)$ the difference $\theta_{2}(t) - \theta_{1}(t)$ will increase beyond $\pi$, and by definition will change its sign since the oscillators change orientation. Ultimately, $\theta_{2}(t) - \theta_{1}(t)$ converges to the equilibrium $\subscr{\theta}{stable}$ in the branch where $\theta_{2} - \theta_{1} < 0$. In the configuration space $\mbb T^{2}$ this implies that the distance $\abs{\theta_{2}(t) - \theta_{1}(t)}$ increases to its maximum value $\pi$ and shrinks again, that is, the oscillators are not phase cohesive and revolve once around the circle before converging to the equilibrium manifold. Since $\sin(\subscr{\theta}{stable}) = \sin(\subscr{\theta}{saddle}) = \kappa^{-1}$, strongly coupled oscillators with $\kappa \gg 1$ practically achieve phase synchronization from every initial condition in an open semi-circle. 
In the critical case, $\kappa = 1$, the saddle point at $\pi/2$ is globally attractive but not stable: for $\theta_{2}(0) - \theta_{1}(0) \!=\! \pi/2 + \epsilon$ (with\,$\epsilon \!>\! 0$ sufficiently small), the oscillators are not phase cohesive and revolve around the circle before converging to the saddle equilibrium manifold\,\,in $\mbb T^{2}$, as illustrated in Figure \ref{Fig: evolution in configuration space}. Thus, the saddle equilibrium\,\,manifold is both attractor and separatrix which\,corresponds to a double zero eigenvalue with two dimensional Jordan block in the linearized\,\,case. 
\begin{figure}[h]
    \centering{
    \subfigure[Vector field \eqref{eq: difference dynamics for two oscillators} for  $\theta_{2} - \theta_{1} > 0$]
    {
    \includegraphics[scale=0.36]{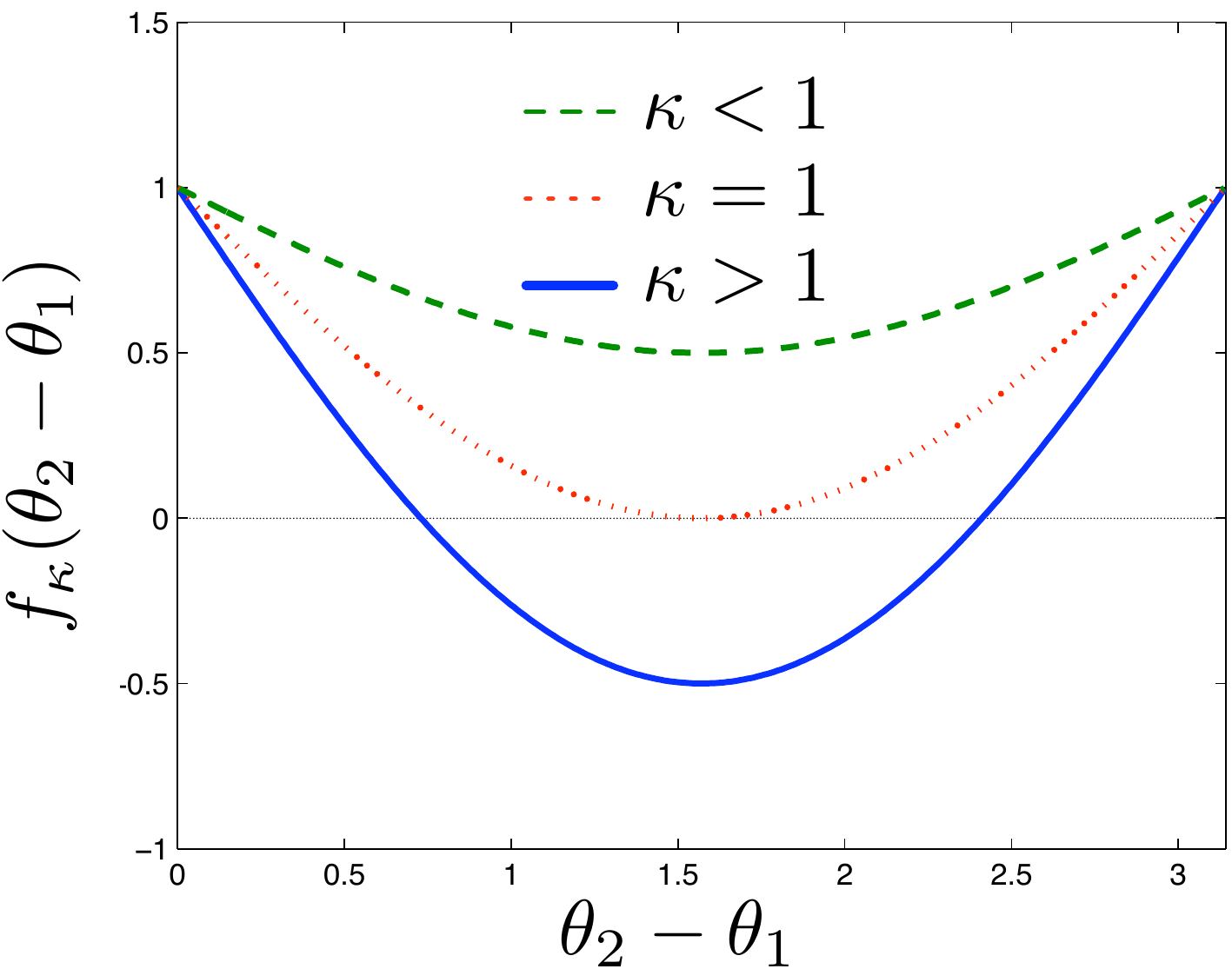}
    \label{Fig: vector field for two oscillators}
    }
   \hspace{1.5cm}
    \subfigure[Trajectory $\theta(t)$ for $\kappa = 1$]% for $\theta_{2} - \theta_{1} > 0$]
    {
    \includegraphics[scale=0.55]{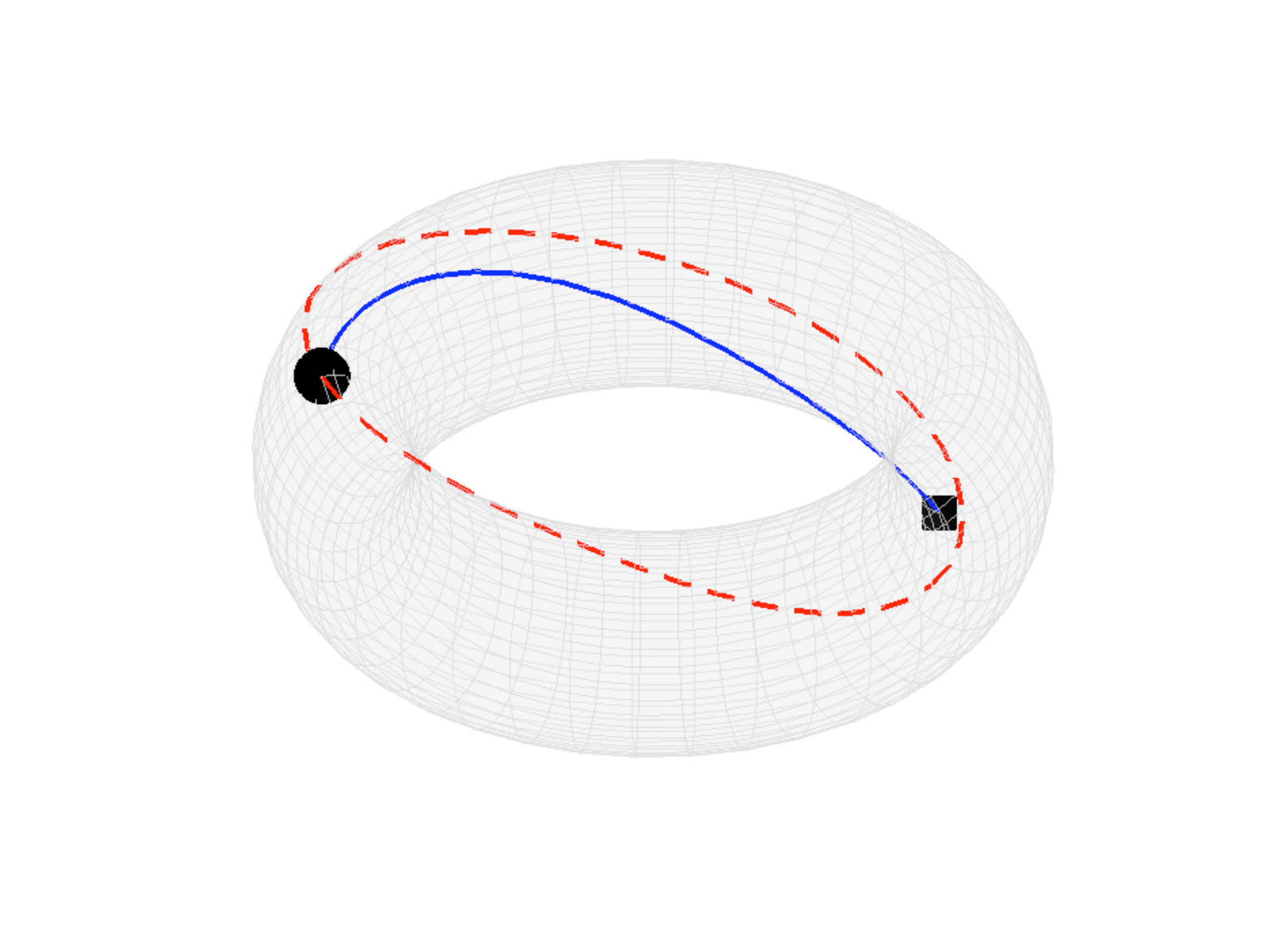}
    \label{Fig: evolution in configuration space}
    }
    \caption{Plot of the vector field \eqref{eq: difference dynamics for two oscillators} for various values of $\kappa$ and a trajectory $\theta(t) \in \mathbb T^{2}$ for the critical case $\kappa = 1$, where the dashed line is the equilibrium manifold and $\blacksquare$ and {\large$\bf\bullet$} correspond to $\theta(0)$ and $\lim_{t \to \infty}\theta(t)$. The non-smoothness of the vector field $f(\theta_{2} - \theta_{1})$ at the boundaries $\{0,\pi\}$ is an artifact of the non-smoothness of the geodesic distance on the state space $\mbb T^{2}$} %and the corresponding potential}% for the case that $\theta_{2} - \theta_{1} > 0$}
    \label{Fig: two oscillators}
    }
\end{figure}

In conclusion, the simple but already rich $2$-dimensional case shows that two oscillators are phase cohesive and synchronize if and only if $K > \subscr{K}{criticial} \triangleq \omega_{2} - \omega_{1}$, and the ratio
$\kappa^{-1} = \subscr{K}{criticial}/K < 1$ determines the ultimate phase cohesiveness as well as the set of admissible initial conditions. In other words, practical phase
synchronization is achieved for $K \gg \subscr{K}{criticial}$, and phase cohesiveness occurs only for initial
angles $\theta(0) \in \bar\Delta(\gamma)$, $\gamma = \arcsin(\subscr{K}{criticial}/K) \in {]\pi/2,\pi[}$.  This
set of admissible initial conditions $\bar\Delta(\gamma)$ can be enlarged to an
open semi-circle by increasing $K/\subscr{K}{criticial}$. 
Finally, synchronization is lost in a saddle-node bifurcation%
\footnote{For Kuramoto models of dimension $n \geq 3$, this loss of synchrony via a saddle-node bifurcation is only the starting point of a series of bifurcation occurring if $K$ is further decreased, see \cite{YLM-OVP-PAT:05}.}
at $K = \subscr{K}{criticial}$.
In Section \ref{Section: New, tight, and explicit bound} we will generalize all outcomes of this simple
analysis to the case of $n$ oscillators.\oprocend
\end{example}

% -----------------------------------------------------------------------------------------------------%
\section{A Review of Bounds for the Critical Coupling Strength}\label{Section: Review of Bounds}
\label{Section: Review of Critical Bounds}
% -----------------------------------------------------------------------------------------------------%

In case that all natural frequencies are identical, that is, $\omega_{i} \equiv \omega$ for all $i \in \until n$, a transformation to a rotating frame leads to $\omega \equiv 0$. In this case, the analysis of the Kuramoto model \eqref{eq: Kuramoto model} is particularly simple and almost global stability can be derived by various methods. A sample of different analysis schemes (by far not complete) 
includes the contraction property \cite{ZL-BF-MM:07}, quadratic Lyapunov functions \cite{AJ-NM-MB:04}, linearization \cite{EC-PM:08}, or order parameter and potential function arguments \cite{RS-DP-NEL:07}. 

In the following, we review various analysis methods and the resulting bounds on the critical coupling strength for the case of non-identical frequencies. %Most of the following techniques also apply to the case of identical frequencies.

% -----------------------------------------------------------------------------------------------------%
\subsection{The Infinite Dimensional Kuramoto Model}

In the physics and dynamical systems communities the Kuramoto model \eqref{eq: Kuramoto model} is typically studied in the continuum limit as the number of oscillators tends to infinity and the natural frequencies obey an integrable distribution function $g: \mbb R \to \mbb R_{\geq 0}$. In this case, the\,\,Kuramoto model is rendered to a first order continuity equation or a second order Fokker-Planck equation when stochasticity is included. For a symmetric, continuous, and unimodal  distribution $g(\omega)$ %of the natural frequencies 
centered above zero, Kuramoto showed in an insightful and ingenuous analysis \cite{YK:75,YK:84} that the incoherent state (i.e., a uniform distribution of\,\,the oscillators on the unit circle) supercritically bifurcates for the critical coupling strength 
\begin{equation}
	\subscr{K}{critical} = \frac{2}{\pi g(0)}
	\label{eq: Kuramoto bound}
	\,.
\end{equation}
The bound \eqref{eq: Kuramoto bound} for the on-set of synchronization has
also been derived by other authors, see \cite{SHS:00,JAA-LLB-CJPV-FR-RS:05}
for further references. In \cite{GBE:85} Ermentrout considered symmetric
distributions $g(\omega)$ with bounded domain $\omega \in
[-\subscr{\omega}{max},\subscr{\omega}{max}]$, and studied the existence of
phase-locked solutions.
% , i.e., $\dot \theta_{i} = \dot \theta_{j}$ for all $i,j \in \until n$.
The condition for the coupling threshold $\subscr{K}{critical}$ necessary
for the existence of phase-locked solution reads in our notation as
\cite[Proposition 2]{GBE:85}
\begin{equation}
	\frac{\subscr{\omega}{max}}{\subscr{K}{critical}}
	=
	\max_{p \in \mbb R, p \geq 1} \left\{ \frac{1}{p^{2}} \int_{-1}^{1} \sqrt{p^{2} - \omega^{2}} g(\omega) d\omega \right\}
	\,.
	\label{eq: Ermentrout bound}
\end{equation}
Ermentrout further showed that formula \eqref{eq: Ermentrout bound} yields $\subscr{K}{critical} \geq 2 \, \subscr{\omega}{max}$ for symmetric
distributions and $\subscr{K}{critical} \geq 4 \subscr{\omega}{max}/\pi$ whenever $g$ is
non-increasing in $[0,\subscr{\omega}{max}]$. Both of these bounds are tight for a bipolar
(i.e., a bimodal double-delta) distribution and a uniform distribution
\cite[Corollary 2]{GBE:85}, \cite[Sections 3 \& 4]{JLvH-WFW:93}. Similar
results for the bipolar distribution are also obtained in
\cite{JAA-LLB-CJPV-FR-RS:05}, and in \cite{EAM-EB-SHS-PS-TMA:09} the critical coupling for a bimodal Lorentzian distribution is analyzed. For various other references analyzing the
continuum limit of the Kuramoto model we refer the reader to \cite{SHS:00,JAA-LLB-CJPV-FR-RS:05}.

% -----------------------------------------------------------------------------------------------------%
\subsection{Necessary or Sufficient Bounds in the Finite Dimensional Kuramoto Model}

In the finite dimensional case, we assume that the natural frequencies are supported on a compact interval $\omega_{i} \in [\subscr{\omega}{max},\subscr{\omega}{min}] \subset \mbb R$ for all $i \in \until n$. This assumption can be made without loss of generality since the critical coupling $\subscr{K}{critical}$ is not finite for unbounded natural frequencies $\omega_{i}$ \cite[Theorem 1]{MV-OM:08}. 
%The general existence of fixed points has been addressed in \cite{JLvH-WFW:93,MV-OM:08,REM-SHS:05,NC-MWS:08,AJ-NM-MB:04}. 
In \cite{NC-MWS:08,AJ-NM-MB:04} a {\it necessary} condition for the existence of synchronized solutions states the critical coupling in terms of the width of the interval $[\subscr{\omega}{max},\subscr{\omega}{min}]$ as
\begin{equation}
	K > \frac{n(\subscr{\omega}{max} - \subscr{\omega}{min})}{2(n-1)}
	\,.
	\label{eq: Jadbabaie bound}
\end{equation}
Obviously, in the limit as $n \!\to\! \infty$, this bound reduces $(\subscr{\omega}{max} - \subscr{\omega}{min})/2$, the simple bound derived in the introduction of this paper. A looser but still insightful necessary condition is $K \geq 2 \sigma$, where $\sigma$  is the variance of the $\omega_{i}$ \cite{JLvH-WFW:93}, \cite[Corollary 2]{MV-OM:08}. For bipolar distributions $\omega_{i} \in \{ \subscr{\omega}{min},\subscr{\omega}{max} \}$, necessary explicit conditions similar to \eqref{eq: Jadbabaie bound} can be derived for non-complete and highly symmetric coupling topologies \cite{LB-LS-ADG:09}.

Besides the necessary conditions, various bounds {\it sufficient} for synchronization have been derived including estimates of the region of attraction. Typically, these sufficient bounds are derived via incremental stability arguments and are of the form
\begin{equation}
	K >
	\subscr{K}{critical}
	=
%	\norm{(\dots, \omega_{i} - \omega_{j},\dots)}_{p} \cdot f(n,\gamma)
	\norm{V\omega}_{p} \cdot f(n,\gamma)
	\label{eq: general sufficient bound}
	\,,
\end{equation}
where $\norm{\cdot}_{p}$ is the $p$-norm and $V$ is a matrix (of yet unspecified row dimension) measuring the non-uniformity among the $\omega_{i}$. For instance, $V = I_{n} - (1/n) \fvec 1_{n \times n}$ gives the deviation from the average natural frequency, $V \omega = \omega - \subscr{\omega}{avg} \fvec 1_{n \times 1}$. Finally, the function $f: \mbb N \times [0,\pi/2[ \to \mbb [1,\infty[$ captures the dependence of $\subscr{K}{critical}$ on the number of oscillators $n$ and the scalar $\gamma$ determining a bound on the admissible pairwise phase differences, which is, for instance, of the form $\norm{(\dots, \theta_{i}(t) - \theta_{j}(t), \dots)}_{p} \leq \gamma$. %, that is, phase cohesiveness in $\subscr{\Delta}{int}(\gamma)$. 

Two-norm bounds, i.e., $p=2$ in condition \eqref{eq: general sufficient bound}, have been derived using quadratic Lyapunov functions in \cite[proof of Theorem 4.2]{NC-MWS:08} and \cite[Theorem V.9]{FD-FB:09o-arxiv}, where the matrix $V \in \mbb R^{n(n-1)/2 \times n}$ is the incidence matrix such that $V \omega$ is the vector of $n(n-1)/2$ pairwise differences $\omega_{i} - \omega_{j}$. A sinusoidal Lyapunov function \cite[Proposition 1]{AF-AC-WPL:10} leads to a two-norm bound with $V = I_{n} - (1/n) \fvec 1_{n \times n}$. Similar two-norm bounds have been obtained by contraction mapping \cite[Theorem 2]{AJ-NM-MB:04} and by contraction analysis \cite[Theorem 3.8]{SJC-JJS:10}, where $V \in \mbb R^{n-1 \times n}$ is an orthonormal projector on the subspace orthogonal to $\fvec 1_{n \times 1}$. For all cited references the region of attraction is given by the $n(n-1)/2$ initial phase differences in two-norm or $\infty$-norm balls satisfying $\norm{V \theta(0)}_{2,\infty} < \pi$. Unfortunately, none of these bounds scales independently of $n$ since $\norm{V \omega}_{2}^{2}$ is a sum of 
%$n(n-1)/2$ \cite{NC-MWS:08,FD-FB:09o-arxiv} $n-1$ \cite{AJ-NM-MB:04,SJC-JJS:10}, or $n$ terms \cite{AF-AC-WPL:10} 
at least $n-1$ terms in all cited references
and $f(n,\gamma)$ in condition \eqref{eq: general sufficient bound} is either an increasing \cite{AJ-NM-MB:04} or a constant function of $n$ \cite{FD-FB:09o-arxiv,NC-MWS:08,SJC-JJS:10,AF-AC-WPL:10}.

A scaling of condition \eqref{eq: general sufficient bound} independently of $n$ has been achieved only when considering the width $\subscr{\omega}{max} -\subscr{\omega}{min} = \norm{(\dots, \omega_{i}-\omega_{j},\dots)}_{\infty}$, that is, for $V \omega$ being the vector of all $n(n-1)/2$ pairwise frequency differences and $p=\infty$ in condition \eqref{eq: general sufficient bound}.
% \cite{NC-MWS:08,GSS-UM-FA:09,FDS-DA:07,SYH-TH-KJH:10,FD-FB:09o-arxiv}. These bounds differ only in the scaling function $f(n,\gamma)$. 
A quadratic Lyapunov function leads to $f(n,\gamma) = n/(2 \sin(\gamma))$ \cite[proof of Theorem 4.1]{NC-MWS:08}, a contraction argument leads to $f(n,\gamma) = n/((n-2)\sin(\gamma))$ \cite[Lemma 9]{GSS-UM-FA:09}, and a geometric argument leads to the scale-free bound $f(\gamma) = 1/(2 \sin(\gamma/2) \cos(\gamma))$ \cite[proof of Proposition 1]{FDS-DA:07}. In \cite[Theorem 3.3]{SYH-TH-KJH:10} and in our earlier work \cite[Theorem V.3]{FD-FB:09o-arxiv}, the simple and scale-free bound $f(\gamma) = 1/\sin(\gamma)$ has been derived by analyticity and contraction arguments. In our notation, the region of attraction for synchronization is in all  cited references \cite{NC-MWS:08,GSS-UM-FA:09,FDS-DA:07,SYH-TH-KJH:10,FD-FB:09o-arxiv} given as $\theta(0) \in \bar\Delta(\gamma)$ for $\gamma \in {[0,\pi/2[}$.

%We note that the bound that we will derive later in Theorem \ref{Theorem: Necessary and Sufficient conditions} improves all the cited sufficient bounds \cite{NC-MWS:08,AJ-NM-MB:04,SJC-JJS:10,GSS-UM-FA:09,FDS-DA:07,AF-AC-WPL:10}, gives a larger guaranteed region of attraction for synchronization than \cite{FD-FB:09o-arxiv,NC-MWS:08,AJ-NM-MB:04,SJC-JJS:10,GSS-UM-FA:09,FDS-DA:07,AF-AC-WPL:10,SYH-TH-KJH:10}, reduces to the bounds derived in the infinite-dimensional case \cite{JAA-LLB-CJPV-FR-RS:05,GBE:85,JLvH-WFW:93}, and turns out to be a necessary and sufficient condition for synchronization when considering all possible distributions for $\omega_{i}$ supported on $[\subscr{\omega}{min},\subscr{\omega}{max}]$.

% -----------------------------------------------------------------------------------------------------%
\subsection{Implicit and Exact Bounds in the Finite Dimensional Kuramoto Model}

Three recent articles \cite{MV-OM:08,REM-SHS:05,DA-JAR:04} independently derived a set of implicit consistency equations for the {\it exact} critical coupling strength $\subscr{K}{critical}$ for which phase-locked solutions exist. Verwoerd and Mason provided the following implicit formulae to compute $\subscr{K}{critical}$ \cite[Theorem 3]{MV-OM:08}:
\begin{equation}
	\subscr{K}{critical} 
	=
	n u^{*} / \sum\nolimits_{i=1}^{n} \sqrt{1 - (\Omega_{i}/u^{*})^{2}}
	\label{eq: Verwoerd bound 1}
	\,,
\end{equation}
where $\Omega_{i} = \omega_{i} - \frac{1}{n} \sum_{j=1}^{n} \omega_{j}$ and $u^{*} \in [\norm{\Omega}_{\infty},2 \norm{\Omega}_{\infty}]$ is the unique solution to %the implicit equation
\begin{equation}
	2 \sum\nolimits_{i=1}^{n} \!\sqrt{1- (\Omega_{i}/u^{*})^{2}} = \sum\nolimits_{i=1}^{n} \! 1/\sqrt{1 - (\Omega_{i}/u^{*})^{2}}
	\label{eq: Verwoerd bound 2}
	\,.
\end{equation}
Verwoerd and Mason also extended their results to bipartite graphs \cite{MV-OM:09} but did not carry out a stability analysis. The formulae \eqref{eq: Verwoerd bound 1}-\eqref{eq: Verwoerd bound 2} can be reduced exactly to the implicit self-consistency equation derived by Mirollo and Strogatz in \cite{REM-SHS:05}  and by Aeyels and Rogge in \cite{DA-JAR:04}, where additionally a local stability analysis is carried out. The stability analysis \cite{REM-SHS:05,DA-JAR:04} in the $n$-dimensional case shows the same sadle-node bifurcation as the two-dimensional Example \ref{Example: Two coupled Kuramoto Oscillators}: for $K<\subscr{K}{critical}$ there exist no phase-locked solutions, for $K>\subscr{K}{critical}$ there exist stable 
%and (in a small range additionally unstable) 
phase-locked solutions, and for $K=\subscr{K}{critical}$ the Jacobians of phase-locked solutions equilibria have a double zero eigenvalue with two-dimensional Jordan block, as illustrated in Example \ref{Example: Two coupled Kuramoto Oscillators}.

In conclusion, in the finite dimensional case various necessary or sufficient explicit bounds on the coupling strength $\subscr{K}{critical}$ are known as well as implicit formulas to compute $\subscr{K}{critical}$ which is provably a threshold for local stability.

% -----------------------------------------------------------------------------------------------------%
\subsection{The Critical Coupling Strength for Second-Order Kuramoto Oscillators}

For $m=n$, $D_{i}=1$, and $M_{i} = M > 0$ the multi-rate Kuramoto model \eqref{eq: multi-rate Kuramoto model} simplifies to a second-order system of coupled oscillators with uniform inertia and unit damping. Such homogenous second-order Kuramoto models have received some attention in the recent literature \cite{YPC-SYH-SBH:11,HAT-AJL-SO:97,HAT-AJL-SO:97b,HH-GSJ-MYC:02,HH-MYC-JY-KSS:99,JAA-LLB-RS:00,JAA-LLB-CJPV-FR-RS:05}. 

In \cite{YPC-SYH-SBH:11} two sufficient synchronization conditions are derived via second-order Gronwall's inequalities resulting in a bound of the form \eqref{eq: general sufficient bound} with $p=\infty$ together with conditions on sufficiently small inertia or sufficiently large inertia \cite[Theorems 5.1 and 5.2]{YPC-SYH-SBH:11}. In \cite[Theorem 4.1 and 4.2]{YPC-SYH-SBH:11} phase synchronization was also found to depend on the inertia, whereas phase synchronization was found to be independent of the inertia in the corresponding continuum limit model  \cite{JAA-LLB-CJPV-FR-RS:05,JAA-LLB-RS:00}. References \cite{HAT-AJL-SO:97b,HAT-AJL-SO:97} observe a discontinuous first-order phase transition (where the incoherent state looses its stability), which is independent of the distribution of the natural frequencies when the inertia $M$ is sufficiently large. This result is also confirmed in \cite{JAA-LLB-CJPV-FR-RS:05,JAA-LLB-RS:00}. In \cite{HH-GSJ-MYC:02} a  second-order Kuramoto model with time delays is analyzed, and a correlation between the inertia and the asymptotic synchronization frequency and asymptotic magnitude of the order parameter magnitude is observed. In \cite{HH-MYC-JY-KSS:99,JAA-LLB-RS:00,JAA-LLB-CJPV-FR-RS:05} it is reported that inertia suppress synchronization for an externally driven or noisy second-order Kuramoto model, and \cite{JAA-LLB-CJPV-FR-RS:05,JAA-LLB-RS:00} explicitly show that the critical coupling $\subscr{K}{critical}$ increases with the inertia $M$ for a Lorentzian or a bi-polar distribution of the natural frequencies.

The cited results \cite{YPC-SYH-SBH:11,HAT-AJL-SO:97,HAT-AJL-SO:97b,HH-GSJ-MYC:02,HH-MYC-JY-KSS:99,JAA-LLB-RS:00,JAA-LLB-CJPV-FR-RS:05} on the inertial effects on synchronization appear conflicting. Possible reasons for this controversy include that the cited articles consider slightly different scenarios (time delays, noise, external forcing), the cited results are only sufficient, the analyses are based on the infinite-dimensional continuum-limit approximation of the finite-dimensional model \eqref{eq: multi-rate Kuramoto model}, and some results stem from insightful but partially incomplete numerical observations and physical intuition.

% -----------------------------------------------------------------------------------------------------%
\section{Necessary and Sufficient Conditions on the Critical Coupling}
\label{Section: New, tight, and explicit bound}
% -----------------------------------------------------------------------------------------------------%

From the point of analyzing or designing a sufficiently strong coupling in the Kuramoto-type applications \cite{SHS:00,JAA-LLB-CJPV-FR-RS:05,RS-DP-NEL:07,KW-PC-SHS:98,FD-FB:09o-arxiv,PAT:03}, the exact formulae \eqref{eq: Verwoerd bound 1}-\eqref{eq: Verwoerd bound 2} to compute the critical coupling have three drawbacks. First, they are implicit and thus not suited for performance or robustness estimates in case of additional coupling strength, e.g., which level of ultimate phase cohesiveness or which magnitude of the order parameter can be achieved for $K = c \cdot\, \subscr{K}{critical}$ with a certain $c>1$. Second, the corresponding region of attraction of a phase-locked equilibrium for a given $K>\subscr{K}{critical}$ is unknown. Third and finally, the particular natural frequencies $\omega_{i}$ (or their distributions) are typically time-varying, uncertain, or even unknown in the applications \cite{SHS:00,JAA-LLB-CJPV-FR-RS:05,RS-DP-NEL:07,KW-PC-SHS:98,FD-FB:09o-arxiv,PAT:03}. In this case, the exact $\subscr{K}{critical}$ needs to be dynamically estimated and re-computed over time, or a conservatively strong coupling $K \!\gg\! \subscr{K}{critical}$ has to be\,\,chosen. 

The following theorem states an explicit bound on the coupling strength together with performance estimates, convergence rates, and a guaranteed semi-global region of attraction for synchronization. Besides improving all other bounds known to the authors, our bound is tight and thus necessary and sufficient when considering arbitrary distributions of the natural frequencies supported on a compact\,interval. %Furthermore, our explicit bound allows to derive performance estimates on the asymptotic magnitude of the order parameter, the possible ultimate and the admissible initial level of phase cohesiveness, as well as the exponential synchronization rates and local stability properties.

\begin{theorem}{\bf(Explicit, necessary, and sufficient synchronization condition)}
\label{Theorem: Necessary and Sufficient conditions}
Consider the Kuramoto model \eqref{eq: Kuramoto model} with natural
frequencies $(\omega_1,\dots,\omega_n)$ and coupling strength $K$.
% supported on the compact interval$[\subscr{\omega}{min},\subscr{\omega}{max}]$.
The following three statements are equivalent:

\begin{enumerate}
\item[(i)] the coupling strength $K$ is larger than the maximum
  non-uniformity among the natural frequencies, i.e.,
  \begin{equation}
    K > \subscr{K}{critical} \triangleq 
    \subscr{\omega}{max} - \subscr{\omega}{min}
    \label{eq: key-assumption - Kuramoto}
    \;;
  \end{equation}

\item[(ii)] there exists an arc length $\subscr{\gamma}{max}\in
  {]\pi/2,\pi]}$ such that the Kuramoto model \eqref{eq: Kuramoto model}
  synchronizes exponentially for all possible distributions of the natural
  frequencies supported on $[\subscr{\omega}{min},\subscr{\omega}{max}]$
  and for all initial phases $\theta(0) \in \Delta(\subscr{\gamma}{max})$; and
  
%\item[(iii)] there exists a locally exponentially stable synchronized trajectory $\map{\theta}{\real_{\geq0}}{\mbb T^n}$ in $\bar \Delta(\subscr{\gamma}{min})$ for some arc length $\subscr{\gamma}{min}\in {[0,\pi/2[}$ and for all possible distributions of the natural frequencies supported on $[\subscr{\omega}{min},\subscr{\omega}{max}]$.

\item[(iii)]  there exists an arc length $\subscr{\gamma}{min} \in {[0,\pi/2[}$ such that the Kuramoto model \eqref{eq: Kuramoto model} has a locally exponentially stable synchronized trajectory in $\bar\Delta(\subscr{\gamma}{min})$ for all possible distributions of the natural frequencies supported on $[\subscr{\omega}{min},\subscr{\omega}{max}]$.

\end{enumerate}

\smallskip 

If the three equivalent cases (i), (ii), and (iii) hold, then the ratio $\subscr{K}{critical}/K$ and the arc lengths $\subscr{\gamma}{min} \in {[0,\pi/2[}$ and $\subscr{\gamma}{max} \in {]\pi/2,\pi]}$ are related uniquely via $\sin(\subscr{\gamma}{min}) = \sin(\subscr{\gamma}{max}) = {\subscr{K}{critical}}/K$, and the following statements hold: \smallskip

\begin{enumerate}

\item[1)] {\bf phase cohesiveness:} the set $\bar\Delta(\gamma)$ is positively invariant for every $\gamma \in [\subscr{\gamma}{min},\subscr{\gamma}{max}]$,\,and each trajectory starting in $\Delta(\subscr{\gamma}{max})$ approaches asymptotically $\bar\Delta(\subscr{\gamma}{min})$;
	
\item[2)] {\bf order parameter:} the asymptotic value of the magnitude of the order parameter denoted by $r_{\infty} \triangleq \lim_{t \to \infty} \frac{1}{n} |\sum_{j=1}^{n}
  e^{\mathrm{i}\theta_{j}(t)} |$is bounded as
  \begin{equation*}
    1  \geq r_{\infty} \geq \cos\!\left(\frac{\subscr{\gamma}{min}}{2}\right) = \sqrt{\frac{1+\sqrt{1- (\subscr{K}{critical}/K)^{2}}}{2}}
    \,;
  \end{equation*}  
  
  \item[3)] {\bf frequency synchronization:} the asymptotic synchronization frequency is
  the average frequency $\subscr{\omega}{avg} = \frac{1}{n}
  \sum_{i=1}^{n} \omega_{i}$, and, given phase cohesiveness in
  $\bar\Delta(\gamma)$ for some fixed $\gamma <\pi/2$, the exponential
  synchronization rate is no worse than $\subscr{\lambda}{fs} = K
  \cos(\gamma)$; and

\item[4)] {\bf phase synchronization:} if $\omega_{i} = s \in \mathbb R$ for all $i \in
  \until n$, then for every $\theta(0) \in \bar\Delta(\gamma)$, $\gamma \in
  {[0,\pi[}$, the phases synchronize exponentially to the average phase
  $\subscr{\theta}{avg}(t) := \frac{1}{n} \sum_{i=1}^{n} \theta(0) + s \cdot t \pmod{2\pi}$ and the exponential synchronization rate is no worse than
  $\subscr{\lambda}{ps} = K \sinc(\gamma)$.

\end{enumerate}
\end{theorem}

\smallskip

To compare the bound \eqref{eq: key-assumption - Kuramoto} to the bounds presented\,\,in Section \ref{Section: Review of Bounds}, we note from the proof of Theorem \ref{Theorem: Necessary and Sufficient conditions}\,\,that our bound \eqref{eq: key-assumption - Kuramoto} can be equivalently stated as $K > (\subscr{\omega}{max} - \subscr{\omega}{min})/\sin(\gamma)$ and thus improves the sufficient bounds \cite{NC-MWS:08,AJ-NM-MB:04,SJC-JJS:10,GSS-UM-FA:09,FDS-DA:07,AF-AC-WPL:10}. In the simple case $n = 2$ analyzed in Example \ref{Example: Two coupled Kuramoto Oscillators}, the bound \eqref{eq: key-assumption - Kuramoto} is obviously exact and also equals the necessary bound \eqref{eq: Jadbabaie bound}. Furthermore, Theorem \ref{Theorem: Necessary and Sufficient conditions} fully generalizes the observations in  Example \ref{Example: Two coupled Kuramoto Oscillators} to the $n$-dimensional case. In the infinite-dimensional case the bound \eqref{eq: key-assumption - Kuramoto} is tight with respect to the necessary bound for a bipolar distribution $\omega_{i} \in \{\subscr{\omega}{min},\subscr{\omega}{max}\}$ derived in  \cite{JAA-LLB-CJPV-FR-RS:05,GBE:85,JLvH-WFW:93}. Note that condition \eqref{eq: key-assumption - Kuramoto} guarantees synchronization for arbitrary distributions of $\omega_{i}$ supported in $[\subscr{\omega}{min},\subscr{\omega}{max}]$, which can possibly be uncertain, time-varying (addressed in detail in Subsection \ref{Subsection: Extension to time-varying natural frequencies}), or even unknown. Additionally, Theorem \ref{Theorem: Necessary and Sufficient conditions} also guarantees a larger region of attraction $\theta(0) \in \Delta(\subscr{\gamma}{max})$ for synchronization than \cite{NC-MWS:08,AJ-NM-MB:04,SJC-JJS:10,GSS-UM-FA:09,FDS-DA:07,AF-AC-WPL:10,FD-FB:09o-arxiv,SYH-TH-KJH:10}.

Besides the necessary and sufficient bound \eqref{eq: key-assumption -
  Kuramoto}, Theorem \ref{Theorem: Necessary and Sufficient conditions}
gives guaranteed exponential convergence rates for frequency and phase
synchronization, and it establishes a practical stability result in the
sense that the multiplicative gap $\subscr{K}{critical}/K$ in the bound
\eqref{eq: key-assumption - Kuramoto} determines the admissible initial and
the guaranteed ultimate phase cohesiveness as well as the guaranteed
asymptotic magnitude $r$ of the order parameter. In view of this result,
the convergence properties of the Kuramoto model \eqref{eq: Kuramoto model}
are best described by the control-theoretical terminology ``practical phase
synchronization.'' 

The proof of Theorem \ref{Theorem: Necessary and Sufficient conditions}
relies on a contraction argument in combination with a consensus analysis
to show that (i) implies (ii) and thus also 1) - 4) for all natural
frequencies supported on $[\subscr{\omega}{min},\subscr{\omega}{max}]$. In
order to prove the implication (ii) $\implies$ (i), we show that the bound
\eqref{eq: key-assumption - Kuramoto} is tight: if (i) is not satisfied, then
exponential synchronization cannot occur for a bipolar distribution of the
natural frequencies. Finally, the equivalence (i), (ii) $\Leftrightarrow$ (iii) follows from the definition of exponential synchronization and by basic arguments from ordinary differential equations

\begin{proof}{\bf Sufficiency (i) $\implies$ (ii):}
We start by proving the positive invariance of $\bar\Delta(\gamma)$, that is, phase cohesiveness\,in $\bar\Delta(\gamma)$ for some $\gamma \in {[0,\pi]}$. Recall the geodesic distance on\,the torus $\mathbb{T}^1$ and define the\,non-smooth\,function\,$\map{V}{\mathbb{T}^n}{[0,\pi]}$,
% by 
\begin{equation*} 
  	V(\psi) =  \max\setdef{|\psi_i-\psi_j|}{i,j\in\until{n}}.
\end{equation*} 
%A simple geometric argument shows the following fact: if $V(\psi)<2\pi/3$, then there exists a unique arc (i.e., a connected subset of $\mathbb{T}^1$) of length $V(\psi)$ containing $\psi_1,\dots,\psi_n$. 
The arc containing all initial phases has two boundary points: a counterclockwise maximum and a counterclockwise minimum. If we let $\subscr{I}{max}(\psi)$ (respectively $\subscr{I}{min}(\psi)$) denote the set indices of the angles $\psi_1,\dots,\psi_n$ that are equal to the counterclockwise maximum (respectively the counterclockwise minimum), then we may write
\begin{equation*}
  	V(\psi) = |\psi_{m'} - \psi_{\ell'}|,   \;\; \text{for all } 
  	m'\in\subscr{I}{max}(\psi)   \text{ and }
  	\ell'\in\subscr{I}{min}(\psi).
\end{equation*}
By assumption, the angles $\theta_i(t)$ belong to the set $\bar\Delta(\gamma)$ at time $t=0$. We aim to show that they remain so for all subsequent times $t>0$. Note that $\theta(t)\in\bar\Delta(\gamma)$ if and only if $V(\theta(t))\leq \gamma \leq \pi$. Therefore, $\bar\Delta(\gamma)$ is positively invariant if and only if $V(\theta(t))$ does not increase at any time $t$ such that $V(\theta(t))= \gamma$. 
%We exclude the case $V(\theta(t)) = \gamma = 0$, that is, all phases are in synchrony (i.e., identical), since either they remain in synchrony for all time or there is a $\tau>t$ such that $V(\theta(\tau)) > 0$. 
The {\it upper Dini derivative} of $V(\theta(t))$ along the dynamical system~\eqref{eq: Kuramoto model} is given by \cite[Lemma 2.2]{ZL-BF-MM:07}
\begin{equation*}
	D^{+} V (\theta(t))
	=
	\lim_{h \downarrow 0}\sup \frac{V(\theta(t+h)) - V(\theta(t))}{h}
	=
	\dot{\theta}_{m}(t) - \dot{\theta}_{\ell}(t)
	\,,
\end{equation*}
where $m \in \subscr{I}{max}(\theta(t))$ and $\ell \in \subscr{I}{min}(\theta(t))$ are indices with the properties that 
$\dot{\theta}_m(t) = \max \setdef{ \dot{\theta}_{m'}(t) }{m'\in \subscr{I}{max}(\theta(t))}$
and
$\dot{\theta}_\ell(t) = \min 	\setdef{ \dot{\theta}_{\ell'}(t) }{\ell'\in \subscr{I}{min}(\theta(t))}$.
%\begin{gather*}
%  	\dot{\theta}_m(t) = \max \setdef{ \dot{\theta}_{m'}(t) }{m'\in
%    	\subscr{I}{max}(\theta(t))}, %%
%  	\enspace \text{and } \enspace \\%%
%  	\dot{\theta}_\ell(t) = \min
%  	\setdef{ \dot{\theta}_{\ell'}(t) }{\ell'\in \subscr{I}{min}(\theta(t))}.
%\end{gather*}
Written out in components $D^{+} V (\theta(t))$ takes the form
\begin{equation*}
	D^{+} V (\theta(t))
	=
	\omega_{m} - \omega_{\ell} - \frac{K}{n} \sum_{i=1}^{n} \bigl(  \sin(\theta_{m}(t) - \theta_{i}(t)) 
	+ \sin(\theta_{i}(t)-\theta_{\ell}(t)) \bigr)
	\,.
\end{equation*}
Note that the index $i$ in the upper sum can be evaluated\,for $i \in \until n$, and for $i=m$ and $i=\ell$ one of the two sinusoidal terms is zero and the other one achieves its maximum value in $\bar\Delta(\gamma)$. In the following we apply classic trigonometric arguments from the Kuramoto literature \cite{NC-MWS:08,GSS-UM-FA:09,FDS-DA:07}.  The trigonometric identity $\sin(x) + \sin(y) = 2 \sin(\frac{x+y}{2}) \cos(\frac{x-y}{2})$ leads to %the further simplifications
\begin{multline}
	D^{+} V (\theta(t))
	=
	\omega_{m} - \omega_{\ell} - 
	\frac{K}{n} \sum_{i=1}^{n} \left(  2\, \sin\!\left(\frac{\theta_{m}(t) - \theta_{\ell}(t)}{2}\right) \right.
	\\\times
	\left.\cos\!\left(\frac{\theta_{m}(t) - \theta_{i}(t)}{2} - \frac{\theta_{i}(t) - \theta_{\ell}(t)}{2} \right) \right)
	\label{eq: evolution of V(theta(t))}
	\,.
\end{multline}
The equality $V(\theta(t)) = \gamma$ implies that, measuring distances counterclockwise and modulo additional terms equal to multiples of $2\pi$, we have 
$\theta_{m}(t) - \theta_{\ell}(t) \!=\! \gamma$, 
$0 \!\leq\! \theta_{m}(t) - \theta_{i}(t) \!\leq\! \gamma$, and 
$0 \!\leq\! \theta_{i}(t) - \theta_{\ell}(t) \!\leq\! \gamma$. 
Therefore, $D^{+}V(\theta(t))$ simplifies\,\,to
\begin{equation*}
	D^{+} V (\theta(t))
	\leq
	\omega_{m} - \omega_{\ell} - \frac{K}{n} \sum_{i=1}^{n} \left(  2 \sin\Bigl( \frac{\gamma}{2} \Bigr) \cos\Bigl( \frac{\gamma}{2} \Bigr) \right)
	.
\end{equation*}
Reversing the identity from above as $2 \sin(x) \cos(y) = \sin(x-y) + \sin(x+y)$ yields %then the simple expression
\begin{equation*}
	D^{+} V (\theta(t))
	\leq
	\omega_{m} - \omega_{\ell} - \frac{K}{n} \sum_{i=1}^{n}  \sin(\gamma)
	=
	\omega_{m} - \omega_{\ell} - K \sin(\gamma)
	\,.
\end{equation*}
It follows that the length of the arc formed by the angles is non-increasing in $\bar\Delta(\gamma)$ if for any pair $\{m,\ell\}$ it holds that
$
K \sin(\gamma)
\geq
\omega_{m} - \omega_{\ell}
$,
which is true if and only if
\begin{equation}
	K \sin(\gamma)
	\geq
	\subscr{K}{critical}
	\label{eq: critical bound}
	\,,
\end{equation}
where $\subscr{K}{critical}$ is as stated in equation \eqref{eq: key-assumption - Kuramoto}. For $\gamma \in {[0,\pi]}$ the left-hand side of \eqref{eq: critical bound} is a concave function of $\gamma$ that achieves its maximum at $\gamma^{*}=\pi/2$. Therefore, there exists an open set of arc lengths $\gamma \in {[0,\pi]}$ satisfying equation \eqref{eq: critical bound} if and only if equation \eqref{eq: critical bound} is true with the strict equality sign at $\gamma^{*}=\pi/2$, which corresponds to equation \eqref{eq: key-assumption - Kuramoto} in the statement of Theorem \ref{Theorem: Necessary and Sufficient conditions}. Additionally, if these two equivalent  statements are true, then there exists a unique $\subscr{\gamma}{min}\in {[0,\pi/2[}$ and a $\subscr{\gamma}{max}\in {]\pi/2,\pi]}$ that satisfy equation \eqref{eq: critical bound} with the equality sign, namely $\sin(\subscr{\gamma}{min}) = \sin(\subscr{\gamma}{max}) = {\subscr{K}{critical}}/K$. For every $\gamma \in {[\subscr{\gamma}{min},\subscr{\gamma}{max}]}$ it follows that the arc-length $V(\theta(t))$ is non-increasing, and it is strictly decreasing for $\gamma \in {]\subscr{\gamma}{min},\subscr{\gamma}{max}[}$. Among other things, this shows that statement (i) implies statement 1).

The frequency dynamics of the Kuramoto model \eqref{eq: Kuramoto model} can be obtained by differentiating the Kuramoto model \eqref{eq: Kuramoto model} as
\begin{equation}
	\dt \dot{\theta_{i}}
	=
	\sum\nolimits_{j=1}^{n} a_{ij}(t) \,(\dot{\theta}_{j} - \dot{\theta}_{i})
	\label{eq: consensus system for dot theta}
	\,,
\end{equation}
where $a_{ij}(t) = (K/n) \cos(\theta_{i}(t)-\theta_{j}(t))$. In the case that $K > \subscr{K}{critical}$, we just proved that for every $\theta(0) \in \Delta(\subscr{\gamma}{max})$ and for all $\gamma \in {]\subscr{\gamma}{min},\subscr{\gamma}{max}]}$ there exists a finite time $T \geq 0$ such that $\theta(t) \in \bar\Delta(\gamma)$ for all $t \geq T$, and consequently, the terms $a_{ij}(t)$ are strictly positive for all $t \geq T$. Notice also that system \eqref{eq: consensus system for dot theta} evolves on the tangent space of $\mbb T^{n}$, that is, the Euclidean space $\mbb R^{n}$. Now fix $\gamma \in {]\subscr{\gamma}{min},\pi/2[}$ and let $T \geq 0$ such that $a_{ij}(t)> 0$ for all $t \geq T$, and note that the frequency dynamics \eqref{eq: consensus system for dot theta} can be analyzed as the linear time-varying consensus system 
\begin{equation*}
	\dt\dot \theta
	=
	- L(t) \dot \theta
	\,,
\end{equation*}
where $L(t) = \diag(\sum_{j \neq i}^{n} a_{ij}(t)) - A(t))$ is a symmetric, fully populated, and time-varying Laplacian matrix corresponding to the graph induced by $A(t)$.\,\,%
%Note that at each time instance the matrix $-L(t)$ is Metzler with zero row sums, and the weights $a_{ij}(t)$ are bounded continuous functions of time that induce integrated over any non-zero time interval a graph with non-degenerate weights and a globally reachable node. Thus, \eqref{eq: consensus system for dot theta} is a {\it consensus protocol} for the velocities $\dot{\theta}_{i}$, and it follows from the {\it contraction property} \cite[Theorem 1]{LM:04-arxiv} that $ \dot{\theta}_{i}(t) \in [\subscr{\dot{\theta}}{min}(0) ,
%\subscr{\dot{\theta}}{max}(0)]$ for all $t \geq 0$, and the dynamics \eqref{eq: Kuramoto model} converge exponentially to the set, where $\dot{\theta}_{i} = \dot{\theta}_{j}$ for all $i,j \in {\until n}$. This concludes the proof of the statement (i) $\implies$ (ii).\\
For each time instant $t \geq T$, the weights $a_{ij}(t)$ are strictly positive, bounded, and continuous functions of time. Consequently, for each $t \geq T$ the graph corresponding to $L(t)$ is always complete and connected. Thus, for each $t \geq 0$ the unique eigenvector corresponding to the zero eigenvalue is $\fvec 1_{n \times 1}$ and $ \fvec 1_{n \times 1}^{T} \dt \dot{\theta} = 0 $. It follows  that $\sum_{i=1}^{n} \dot \theta_{i}(t) = \sum_{i=1}^{n} \omega_{i} = n \subscr{\omega}{avg}$ is a conserved quantity. Consider the {\it disagreement vector}
$
\dot \delta
=
\dot{\theta} - \subscr{\omega}{avg} \fvec 1_{n \times 1}
$, 
as an error coordinate satisfying 
$
\fvec 1_{n \times 1}^{T} \dot \delta 
%=
%\fvec 1^{T} \dot{\theta} - \fvec 1^{T} \subscr{\omega}{avg} \fvec 1
= 
0
$,
that is, $\dot\delta$ lives in the {\it disagreement eigenspace} of dimension $n-1$ with normal vector $\fvec 1_{n \times 1}$. Since $\subscr{\omega}{avg}$ is constant and $\ker(L(t)) \!\equiv\! \mathrm{span}(\fvec 1_{n \times 1})$, the dynamics \eqref{eq: consensus system for dot theta} read in
$\dot \delta$-coordinates\,\,as
\begin{equation}
	\dt \dot\delta
	=
	- L(t) \, \dot\delta
	\label{eq: consensus system for dot delta-disagreement}
	\,.
\end{equation}
Consider the {\it disagreement function} $\dot\delta \mapsto \|\dot\delta\|^{2} = \dot\delta^{T} \dot\delta$ and its derivative along the disagreement dynamics \eqref{eq: consensus system for dot delta-disagreement} which is
$
\dt
\|\dot\delta\|^{2}
=
- 2 \, \dot \delta^{T} L(t) \dot\delta
$. By the Courant-Fischer Theorem, the time derivative of the disagreement function can be upper-bounded (point-wise in time) by the second-smallest eigenvalue of the Laplacian $L(t)$, i.e., the algebraic connectivity $\lambda_{2}(L(t))$, as $\dt \|\dot\delta\|^{2} \leq -2\lambda_{2}(L(t)) \|\dot\delta\|^{2}$. The algebraic connectivity $\lambda_{2}(L(t))$ can be lower-bounded as
$
	\lambda_{2}(L(t))
	\geq
	%%\\
	K \min\nolimits_{i,j \in \until n}\{\cos(\theta_{i}-\theta_{j}) |\, \theta \in \bar\Delta(\gamma) \} 
	%%\\
	\geq
	K \cos(\gamma)
	= \subscr{\lambda}{fs}
	\,.
$
Thus, the derivative of the disagreement function is bounded as
$
\dt
\|\dot \delta\|
\leq
-
2\,\subscr{\lambda}{fs}
\|\dot \delta\|^{2}
$.
The Bellman-Gronwall Lemma \cite[Lemma A.1]{HKK:02}  yields that the disagreement
vector $\delta(t)$ satisfies 
$ 
\|\dot\delta(t)\| 
\!\leq\! 
\|\dot\delta(0)\| e^{- \subscr{\lambda}{fs} t} 
$ 
for all $t \geq  T$. This proves statement 3) and concludes the proof of the sufficiency (i) $\implies$ (ii).

%\bigskip 
{\bf Necessity (ii) $\implies$ (i):} To show that the critical
coupling in condition \eqref{eq: key-assumption - Kuramoto} is also necessary for
synchronization, it suffices to construct a counter example for which $K \leq \subscr{K}{critical}$ and the oscillators do not achieve exponential synchronization even though all $\omega_{i} \in {[\subscr{\omega}{min},\subscr{\omega}{max}]}$ and $\theta(0) \in \Delta(\gamma)$ for every $\gamma \in {]\pi/2,\pi]}$. % for every $\theta(0) \in \Delta(\subscr{\gamma}{max})$. 
    A basic instability mechanism under which synchronization breaks down is caused by a bipolar
distribution of the natural frequencies, as shown in Example \ref{Example: Two
  coupled Kuramoto Oscillators}.

% Once the basic instability mechanism in the example is understood, a more
% general counter-example can be constructed as follows.

Let the index set $\until n$ be partitioned by the two non-empty sets $\mc I_{1}$ and $\mc I_{2}$. Let $\omega_{i} = \subscr{\omega}{min}$ for $i \in \mc I_{1}$ and $\omega_{i} = \subscr{\omega}{max}$ for $i \in \mc I_{2}$, and assume that at some time $t \geq 0$ it holds that $\theta_{i}(t) \!=\! - \gamma/2$ for $i \in \mc I_{1}$ and $\theta_{i}(t) \!=\! + \gamma/2$ for $i \in \mc I_{2}$ and for some $\gamma \in {[0,\pi[}$. By construction, at time $t$ all oscillators are contained in an arc of length $\gamma \in {[0,\pi[}$. Assume now that $K\!<\!\subscr{K}{critical}$ and the oscillators synchronize. Consider the evolution of the arc length $V(\theta(t))$ given as in \eqref{eq: evolution of V(theta(t))} by
\begin{align*}
	D^{+} V (\theta(t))
	=&\;
%	\omega_{m} - \omega_{\ell} - \frac{K}{n} \sum_{i=1}^{n} \left(  2\, \sin\!\left(\frac{\theta_{m}(t) - \theta_{\ell}(t)}{2}\right) 
%	\cos\!\left(\frac{\theta_{m}(t) - \theta_{i}(t)}{2} - \frac{\theta_{i}(t) - \theta_{\ell}(t)}{2} \right) \right)	
%	\\=&\;
	\omega_{m} - \omega_{\ell} - \frac{K}{n} \sum_{i \in \mc I_{1}} \left(  2\, \sin\!\left(\frac{\theta_{m}(t) - \theta_{\ell}(t)}{2}\right) \right.
	\\&\times
	\left.\cos\!\left(\frac{\theta_{m}(t) - \theta_{i}(t)}{2} - \frac{\theta_{i}(t) - \theta_{\ell}(t)}{2} \right) \right)	
	\\&\;	
	-  \frac{K}{n} \sum_{i \in \mc I_{2}} \left(  2\, \sin\!\left(\frac{\theta_{m}(t) - \theta_{\ell}(t)}{2}\right) \,
	 \cos\!\left(\frac{\theta_{m}(t) - \theta_{i}(t)}{2} - \frac{\theta_{i}(t) - \theta_{\ell}(t)}{2} \right) \right)	
	\,,
\end{align*}
where the summation is split according to the partition of $\until n$ into $\mc I_{1}$ and $\mc I_{2}$. 
%By construction of the natural frequencies $\omega_{i}$ and the state at time $t$, 
By construction, we have that $\ell \in \mc I_{1}$, $m \in \mc I_{2}$, $\omega_{\ell} = \subscr{\omega}{min}$, $\omega_{m} = \subscr{\omega}{max}$, $\theta_{i}(t) = \theta_{\ell}(t) = - \gamma/2$ for $i \in \mc I_{1}$, and $\theta_{i}(t) = \theta_{m}(t) = +\gamma/2$ for $i \in \mc I_{2}$. Thus, $D^{+} V (\theta(t))$ simplifies to
\begin{equation*}
	D^{+} V (\theta(t))
	=
	\subscr{\omega}{max} - \subscr{\omega}{min} - \frac{K}{n} \sum_{i \in \mc I_{1}} \left(  2 \sin\Bigl( \frac{\gamma}{2} \Bigr) \cos\Bigl( \frac{\gamma}{2} \Bigr) \right)
	-  \frac{K}{n} \sum_{i \in \mc I_{2}} \left(  2 \sin\Bigl( \frac{\gamma}{2} \Bigr) \cos\Bigl( \frac{\gamma}{2} \Bigr) \right)
	\,.
\end{equation*}
Again, we reverse the trigonometric identity via $2 \sin(x) \cos(y) =$ $ \sin(x-y) + \sin(x+y)$, unite both sums, and arrive at
\begin{equation}
	D^{+} V (\theta(t))
	=
	\subscr{\omega}{max} - \subscr{\omega}{min} - K \sin(\gamma)
	\label{eq: critical bound - necessary}
	\,.
\end{equation}
Clearly, for $K < \subscr{K}{critical}$ the arc length $V(\theta(t)) = \gamma$ is increasing for any arbitrary $\gamma \in {[0,\pi]}$. Thus, the phases are not bounded in $\bar\Delta(\gamma)$. This contradicts the assumption that the oscillators synchronize for $K < \subscr{K}{critical}$ from every initial condition $\theta(0) \in \bar\Delta(\gamma)$. Thus, $\subscr{K}{critical}$ provides the exact threshold. For $K = \subscr{K}{critical}$, we know from \cite{REM-SHS:05,DA-JAR:04} that phase-locked equilibria have a zero eigenvalue with a two-dimensional Jacobian block, and thus synchronization cannot occur. This instability via a two-dimensional Jordan block is also visible in \eqref{eq: critical bound - necessary} since $D^{+} V (\theta(t))$ is increasing for $\theta(t) \in \Delta(\gamma)$, $\gamma \in {]\pi/2,\pi]}$ until all oscillators change orientation, just as in Example \ref{Example: Two coupled Kuramoto Oscillators}. This concludes the proof of the necessity (ii)\,$\implies$\,(i).

{\bf Sufficiency (i),(ii) $\implies$ (iii):} Assume that (i) and (ii) hold and exponential synchronization occurs. When formulating the Kuramoto model \eqref{eq: Kuramoto model} in a rotating frame with frequency $\subscr{\omega}{avg}$, statement 3) implies exponential convergence of the frequencies $\dot \theta_{i}(t)$ to zero. Hence, for all $\theta(0) \in \Delta(\subscr{\gamma}{max})$ every phase $\theta_{i}(t)$ converges exponentially to a constant limit phase given by $\subscr{\theta}{$i$,sync} \triangleq \lim_{t \to \infty} \theta_{i}(t) = \theta_{i}(0) + \int_{0}^{\infty} \dot \theta_{i}(\tau)$, which corresponds to an equilibrium of the Kuramoto model \eqref{eq: Kuramoto model} formulated in a rotating frame. Furthermore, statement 1) implies that these equilibria $(\subscr{\theta}{$1$,sync},\dots, \subscr{\theta}{$n$,sync})$ are contained in $\bar\Delta(\subscr{\gamma}{min})$. Since the Kuramoto model \eqref{eq: Kuramoto model} features only a finite number of fixed points (modulo translational invariance) \cite[Lemma 1.1 and Theorem 4.1]{JB-CIB:82}, almost every trajectory with $\theta(0) \in \Delta(\subscr{\gamma}{max})$ converges exponentially to a one-dimensional stable equilibrium manifold. Hence,  if condition (i) holds, there exists a locally exponentially stable synchronized solution $\theta(t) \in \bar \Delta(\subscr{\gamma}{min})$. 

{\bf Necessity (iii) $\implies$ (i),(ii):} Conversely, assume that condition (i) does not hold, that is, $K \leq \subscr{K}{critical} = \subscr{\omega}{max} - \subscr{\omega}{min}$. We prove the necessity of (i) again by invoking a bipolar distribution of the natural frequencies. In this case, it is known that for $K = \subscr{K}{critical} = \subscr{\omega}{max} - \subscr{\omega}{min}$ there exists a unique equilibrium (in a rotating frame with frequency $\subscr{\omega}{avg}$), and for $K< \subscr{K}{critical}$ there exists no equilibrium \cite[Section 4]{JLvH-WFW:93}. In the latter case, synchronization cannot occur. In the former case, the equilibrium configuration corresponds to the phases arranged in two clusters (sorted according to the bipolar distribution) which are exactly $\pi/2$ apart \cite[Section 4]{JLvH-WFW:93}. Finally, note that such an equilibrium configuration is unstable, as shown by identity \eqref{eq: critical bound - necessary}. We remark that the same conclusions can alternatively be drawn from the implicit equations \eqref{eq: Verwoerd bound 1}-\eqref{eq: Verwoerd bound 2} for the critical coupling. This proves the necessity (iii)\,$\Rightarrow$\,(i),(ii).
%In this case, the implicit equations \eqref{eq: Verwoerd bound 1}-\eqref{eq: Verwoerd bound 2} yield the critical coupling as $K = \subscr{K}{critical} = \subscr{\omega}{max} - \subscr{\omega}{min}$ (This can be easily verified for an even number $n$ and a balanced bipolar distribution or an odd number $n$ with $n-1$ balanced bipolar frequencies and one zero frequency.).
%%we obtain $u^{*} = \sqrt(2)/2 \times (\subscr{\omega}{max} - \subscr{\omega}{min})$. 
%Hence, no fixed point exists for $K< \subscr{K}{critical} $, and for $K = \subscr{K}{critical}$ the instability via a two-dimensional Jordan block follows again from the arguments in \cite{REM-SHS:05,DA-JAR:04}. This shows the necessity of condition (i) and concludes the proof of statement 5).
% equations \eqref{eq: Verwoerd bound 1}-\eqref{eq: Verwoerd bound 2} yield $u^{*} = \sqrt(2)/2 \times (\subscr{\omega}{max} - \subscr{\omega}{min})$.
%Mirollo and Strogatz \cite{REM-SHS:05} and Aeyles and Rogge \cite{DA-JAR:04} showed that a necessary condition for the existence of a fixed point (in relative coordinates) is $\cos(\subscr{\theta}{$i$,sync} - \psi) \geq 0$ for all $i \in \until n$, where $\psi \in \mbb T$ is the phase of the order parameter. Notice that the condition $\cos(\subscr{\theta}{$i$,sync} - \psi) \geq 0$ for all $i \in \until n$ cannot be satisfied if $\subscr{\theta}{$i$,sync} \not\in \Delta(\pi/2)$ for all $i \in \until n$.

By statement 1), the oscillators are ultimately phase cohesive in $\bar\Delta(\subscr{\gamma}{min})$. It follows from Lemma \ref{Lemma: phase cohesiveness and order parameter} that the asymptotic magnitude $r$ of the order parameter satisfies $1  \geq r \geq \cos(\subscr{\gamma}{min}/2)$. The trigonometric identity $\cos(\subscr{\gamma}{min}/2) = \sqrt{(1+\cos(\subscr{\gamma}{min}))/2}$ together with a Pythagorean identity yields the bound in statement 2).

In case that all natural frequencies are identical, that is, $\omega_{i} = s$ for all $i \in \until n$, statement 1) implies that $\subscr{\gamma}{min} = 0$ and $\subscr{\gamma}{max} \uparrow \pi$. In short, the phases synchronize for every $\theta(0) \in \Delta(\pi)$. The coordinate transformation $\theta \mapsto \theta + s t$ yields the dynamics
$
	\dot \theta_{i}
	=
	- \sum_{j=1}^{n} b_{ij}(t) (\theta_{i} - \theta_{j})
$,
where $b_{ij}(t) = (K/n) \sinc(\theta_{i}(t) - \theta_{j}(t))$ is strictly positive for all $t \geq 0$ due to the positive invariance statement 1). Statement 4) can then be proved along the lines of statement 3).
\end{proof}

% -----------------------------------------------------------------------------------------------------%
\subsection{Statistical studies}\label{Subsection: Statistical studies}

Theorem \ref{Theorem: Necessary and Sufficient conditions} places a hard bound on the critical coupling strength $\subscr{K}{critical}$ for all distributions of $\omega_{i}$ supported on the compact interval $[\subscr{\omega}{max},\subscr{\omega}{min}]$. This set of admissible distributions includes the worst-case bipolar  distribution used in the proof of Theorem \ref{Theorem: Necessary and Sufficient conditions}. For a particular distribution $g(\omega)$ supported on $[\subscr{\omega}{min},\subscr{\omega}{max}]$ the bound \eqref{eq: key-assumption - Kuramoto} is only sufficient and possibly a factor 2 larger than the necessary bound\,\,\eqref{eq: Jadbabaie bound}. The exact critical coupling for $g(\omega)$ lies somewhere in between and can be obtained by solving the implicit equations\,\,\eqref{eq: Verwoerd bound 1}-\eqref{eq: Verwoerd bound 2}.

The following example illustrates the average case for natural frequencies sampled from a uniform distribution $g(\omega) = 1/2$ supported for $\omega \in [-1,1]$. Figure \ref{Fig: Kuramoto bounds for uniform distribution} reports numerical findings on the critical coupling strength for $n \in [2,300]$ oscillators in a semi-log plot, where the coupling gains for each $n$ are averaged over 1000 simulations.

\begin{figure}[htb]
	\centering{
	\includegraphics[scale=0.47]{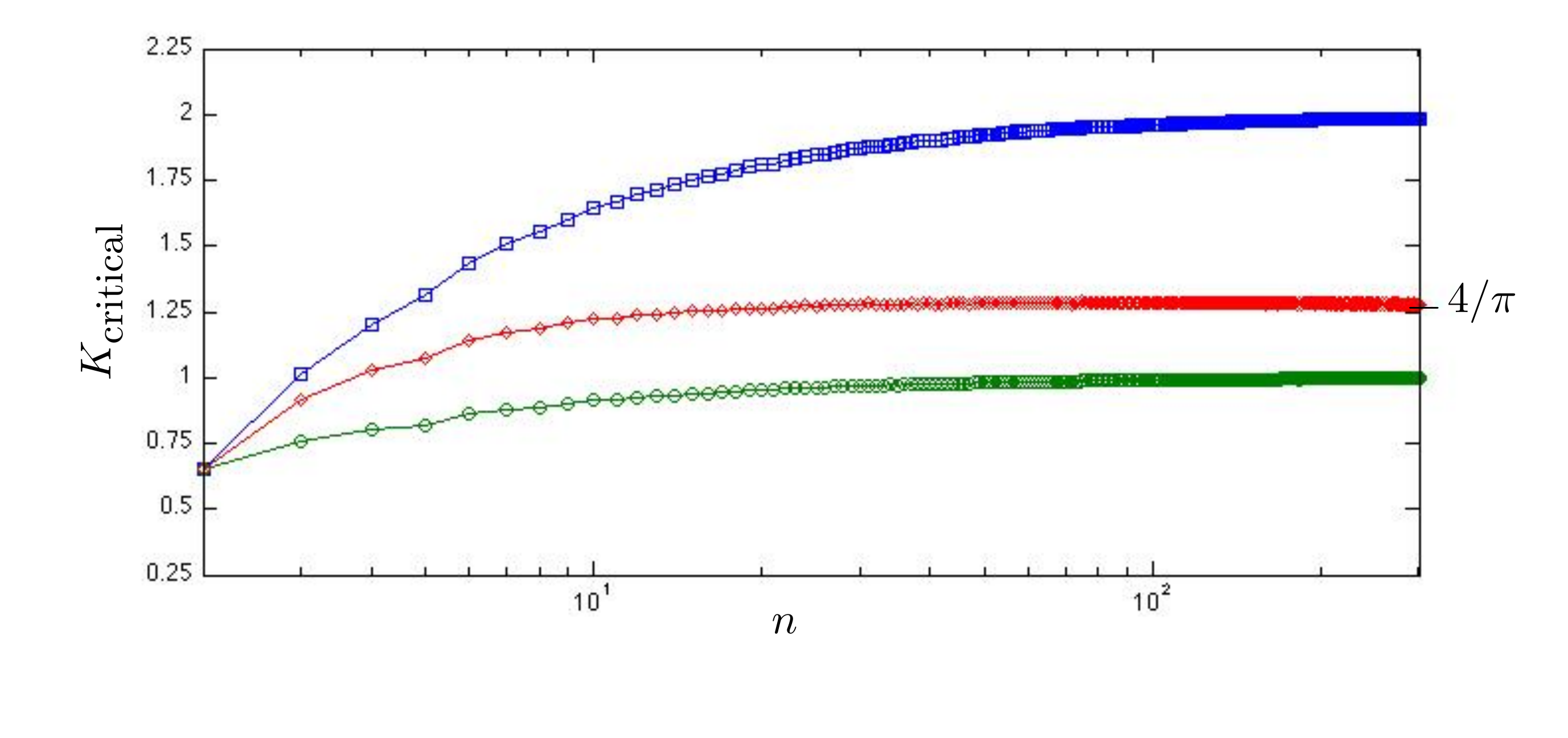}
	\caption{Analysis of the necessary bound \eqref{eq: Jadbabaie bound} ({\bf\large$\color{red}{\circ}$}), the exact bound \eqref{eq: Verwoerd bound 1}-\eqref{eq: Verwoerd bound 2} ({\bf\footnotesize$\color{OliveGreen}\lozenge$}), and the sufficient explicit bound \eqref{eq: key-assumption - Kuramoto} ({\bf\tiny$\color{blue}{\square}$})}% on the critical coupling strength $\subscr{K}{critical}$}
	\label{Fig: Kuramoto bounds for uniform distribution}
	}
\end{figure}
First, note that the three displayed bounds are equivalent for $n=2$ oscillators. As the number of oscillators increases, the sufficient bound \eqref{eq: key-assumption - Kuramoto} clearly converges to $\subscr{\omega}{max}-\subscr{\omega}{min} = 2$, the width of the distribution $g(\omega)$, and the necessary bound \eqref{eq: Jadbabaie bound} accordingly to half of the width. The exact bound \eqref{eq: Verwoerd bound 1}-\eqref{eq: Verwoerd bound 2} quickly converges to $4(\subscr{\omega}{max}-\subscr{\omega}{min})/(2\pi) = 4/\pi$ in agreement with the results \eqref{eq: Kuramoto bound} and \eqref{eq: Ermentrout bound} predicted in the case of a continuum of oscillators. It can be observed that the exact bound \eqref{eq: Verwoerd bound 1}-\eqref{eq: Verwoerd bound 2} is closer to the sufficient and tight bound \eqref{eq: key-assumption - Kuramoto} for a small number of oscillators, i.e., when there are few outliers increasing the width $\subscr{\omega}{max} - \subscr{\omega}{min}$. For large $n$, the sample size of $\omega_{i}$ increases and thus also the number of outliers. In this case, the exact bound \eqref{eq: Verwoerd bound 1}-\eqref{eq: Verwoerd bound 2} is closer to the necessary bound \eqref{eq: Jadbabaie bound}.

% -----------------------------------------------------------------------------------------------------%
\subsection{Extension to time-varying natural frequencies}
\label{Subsection: Extension to time-varying natural frequencies}

One motivation to prefer the explicit and tight bound \eqref{eq: key-assumption - Kuramoto} over the implicit and exact bound \eqref{eq: Verwoerd bound 1}-\eqref{eq: Verwoerd bound 2} are time-varying natural frequencies $\omega_{i}(t)$ bounded in $[\subscr{\omega}{max},\subscr{\omega}{min}]$. We distinguish the two cases of switching and slowly and smoothly time-varying natural frequencies and note that Theorem \ref{Theorem: Necessary and Sufficient conditions} and its proof can be easily extended to these cases.

% -----------------------------------------------------------------------------------------------------%
\subsubsection{\bf Piece-wise constant $\omega_{i}(t)$} 

Consider a sequence of time instances $\{t_{k}\}_{k \in \mbb N}$ such that $t_{0} = 0$ and $t_{k+1} > t_{k}$ for all $k \in \mathbb N$. Assume that the natural frequencies $\omega_{i}(t)$ are constant and bounded in $[\subscr{\omega}{max},\subscr{\omega}{min}]$ within each interval $t \in [t_{k},t_{k+1}[$. At time-point $t_{k+1}$ the natural frequencies may be discontinuous and switch. Note that the synchronization frequency and the corresponding phase-locked equilibria (on a rotating frame) will change with every switching instant. 

In this case, between any two switching instants, $t \in [t_{k},t_{k+1}[$, our analysis still holds and Theorem \ref{Theorem: Necessary and Sufficient conditions} can be applied without any modification. For all time $t \geq 0$ and for all $\theta \in \Delta(\gamma)$, $\gamma \in {]\subscr{\gamma}{min},\subscr{\gamma}{max}]}$, the arc length $V(\theta(t))$ is strictly and uniformly decreasing for {\em any} switching sequence $\{t_{k}\}_{k \in \mbb N}$, i.e, it is a so-called {\em common Lyapunov function}. As an outcome, the ultimate phase cohesiveness in $\bar\Delta(\subscr{\gamma}{min})$ will always be reached asymptotically despite the switching natural frequencies. Furthermore, if there exists a uniform {\em dwell time} $\epsilon > 0$ such that $t_{k+1} - t_{k} \geq \epsilon$ for all $k \in \mbb N$, then the derived synchronization rate $\subscr{\lambda}{fs}$ admits an estimate on $\lim_{t \uparrow t_{k+1}}\|\dot \theta(t) - \subscr{\omega}{avg}(t) \fvec 1_{n \times 1}\|$, that is, how close the oscillators come to frequency synchronization within each interval $[t_{k},t_{k+1}[$. Figure \ref{Fig: switching natural frequencies} illustrates all of these conclusion in a simulation.
\begin{figure}[htbp]
	\centering{
	\includegraphics[scale=0.65]{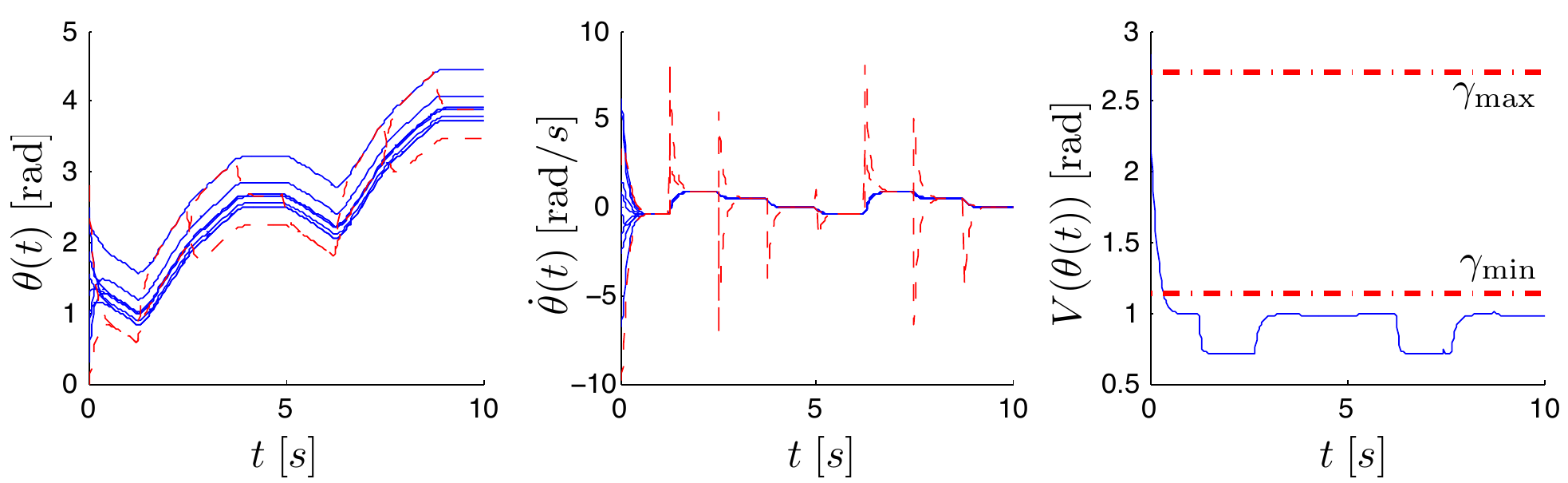}
	\caption{Simulation of a network of $n=10$ Kuramoto oscillators satisfying $K/\subscr{K}{critical} = 1.1$, where the natural frequencies $\omega_{1}(t)$ and $\omega_{n}(t)$ (displayed in red dashed lines) are switching between constant values in $[\subscr{\omega}{min},\subscr{\omega}{max}] = [0,1]$. The simulation illustrates the phase cohesiveness of the angles $\theta(t)$ in $\bar\Delta(\subscr{\gamma}{min})$, the exponential convergence of the frequencies $\dot\theta(t)$ towards $\subscr{\omega}{avg}(t)$ between consecutive switching instances, as well as the monotonicity of $V(\theta(t))$ in $\bar\Delta(\gamma)$ for $\gamma \in [\subscr{\gamma}{min},\subscr{\gamma}{max}]$.}
	\label{Fig: switching natural frequencies}
	}
\end{figure}

In comparison, the analysis schemes \cite{MV-OM:08,DA-JAR:04,REM-SHS:05} have to re-compute the exact implicit bound \eqref{eq: Verwoerd bound 1}-\eqref{eq: Verwoerd bound 2} after every switching instant, since they explicitly make use of the values of $\omega_{i}$ and the corresponding equilibria. 
Obviously, the analysis schemes \cite{MV-OM:08,DA-JAR:04,REM-SHS:05} fail entirely in the case of time-varying frequencies analyzed in the following.

% -----------------------------------------------------------------------------------------------------%
\subsubsection{Slowly and smoothly varying $\omega_{i}(t)$} 

For smooth functions $\omega_{i}(t)$ bounded in $[\subscr{\omega}{max},\subscr{\omega}{min}]$, the proof of phase cohesiveness can be adapted without major modifications. However, the Kuramoto frequency dynamics \eqref{eq: consensus system for dot theta} are rendered to
\begin{equation}
	\dt \dot{\theta_{i}}
	=
	\dot \omega_{i}(t) + 
	\sum\nolimits_{j=1}^{n} a_{ij}(t) \,(\dot{\theta}_{j} - \dot{\theta}_{i})
	\label{eq: consensus system for dot theta -- time-varying omega}
	\,,
\end{equation}
where $a_{ij}(t) = (K/n) \cos(\theta_{i}(t)-\theta_{j}(t))$ as before. The forced frequency dynamics \eqref{eq: consensus system for dot theta -- time-varying omega} can be analyzed on the subspace orthogonal to $\fvec 1_{n \times 1}$ by considering the time-varying disagreement vector 
$
\dot \delta(t)
\triangleq
\dot{\theta}(t) - \subscr{\omega}{avg}(t) \fvec 1_{n \times 1}
$, 
as an error coordinate satisfying 
$
\fvec 1_{n \times 1}^{T} \dot \delta(t) 
%=
%\fvec 1^{T} \dot{\theta} - \fvec 1^{T} \subscr{\omega}{avg} \fvec 1
= 
0
$.
The frequency dynamics \eqref{eq: consensus system for dot theta -- time-varying omega} read then in $\dot \delta$-coordinates as
\begin{equation}
	\dt \dot\delta
	=
	\dot\Omega(t)
	- L(t) \, \dot\delta
	\label{eq: consensus system for dot delta-disagreement -- time-varying omega}
	\,,
\end{equation}
where $\dot\Omega(t) \triangleq \dot \omega(t) - \subscr{\dot \omega}{avg}(t) \fvec 1_{n \times 1}$. On the subspace orthogonal to $\fvec 1_{n \times 1}$ the dynamics \eqref{eq: consensus system for dot delta-disagreement -- time-varying omega} are exponentially stable for $\dot\Omega(t) \equiv 0$, and a {\it time-varying equilibrium frequency} can be uniquely obtained as $\dot\delta(t) = L^{\dagger}(t) \dot\Omega(t)$, where $L^{\dagger}$ is the Moore-Penrose inverse of $L$. In this case, the standard theory of slowly varying systems \cite[Chapter 9.6]{HKK:02} can be applied for a slowly varying $\dot\Omega(t)$ satisfying $\| \ddot\Omega(t)\|_{\infty} \leq \epsilon$ for $\epsilon$ sufficiently small.
% (the particular bound depends on $\lambda_{2}(L(t))$ and the phase cohesiveness $\delta(t) \in \Delta(\gamma)$). It follows that system \eqref{eq: consensus system for dot delta-disagreement -- time-varying omega} can be analyzed with a slowly-time varying equilibrium point $\delta^{*}(t) = L^{\dagger}(t) u(t)$, $\delta(t) - L^{\dagger}(t) u(t)$ is uniformly ultimately bounded (the ultimate bound is proportionally to $\epsilon$) and $\delta(t) - L^{\dagger}(t) u(t) \to 0$ if $\dot u(t) \to 0$ as $t \to \infty$ \cite[Theorem 9.3]{HKK:02}. 

In summary, if each $\omega_{i}(t)$ is a smooth, bounded in $[\subscr{\omega}{max},\subscr{\omega}{min}]$, and the relative acceleration $\|\ddot\Omega(t)\|_{\infty} =\|\ddot \omega(t) - \subscr{\ddot \omega}{avg}(t) \fvec 1_{n \times 1}\|_{\infty} \leq \epsilon$ is sufficiently small, then there\,\,exists $T \geq 0$ and $k=k(\epsilon)>0$ such that the frequencies satisfy $\|\dot\delta(t) - L^{\dagger}(t) \dot\Omega(t)\|_{\infty} \leq k$\,\,for all $t \geq T$. Moreover, if $\ddot\Omega(t) \to 0$ as $t \to \infty$, then $\dot \delta(t) \to L^{\dagger}(t) \dot\Omega(t)$ as $t \to \infty$. In particular, $\epsilon$ and $k$ depend on the phase cohesiveness\,$\delta(t) \in \bar\Delta(\gamma)$, see \cite[Theorem\,9.3]{HKK:02} for details. Figure \ref{Fig: time-varying natural frequencies} illustrates these conclusions with a simulation of $n=10$ oscillators. The authors of \cite{AF-AC-WPL:10,DC-CPU:07} come to a similar conclusion when analyzing the effects of time-varying frequencies via input-to-state stability arguments or in simulations.

\begin{figure}[htbp]
	\centering{
	\includegraphics[scale=0.68]{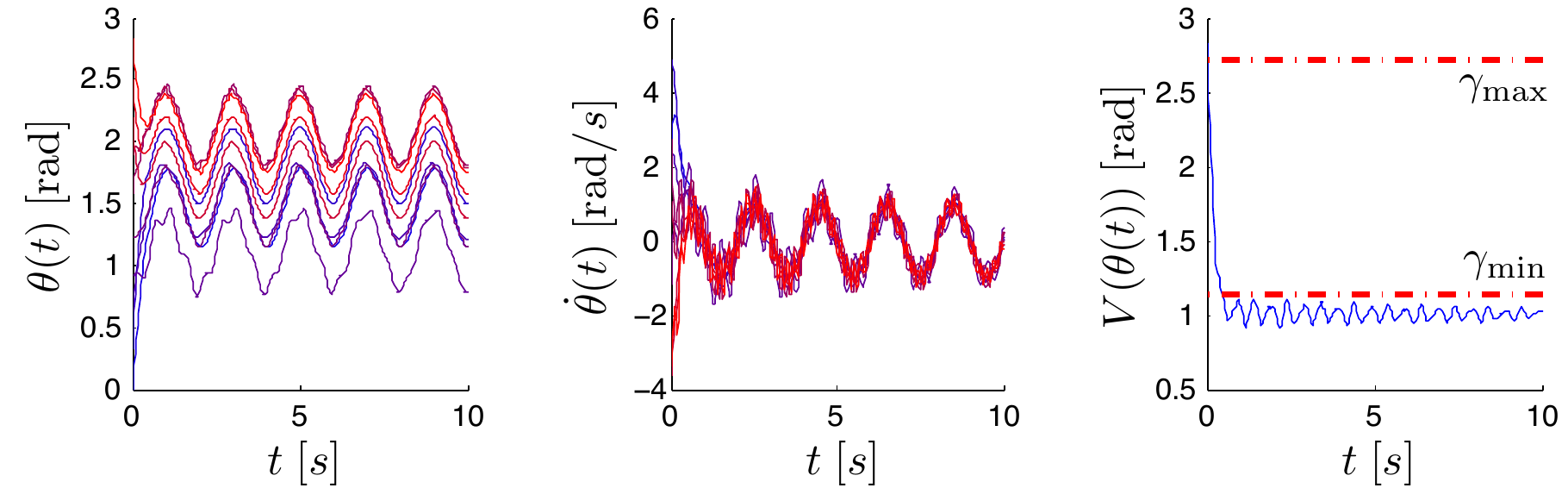}
	\caption{Simulation of a network of $n=10$ Kuramoto oscillators satisfying $K/\subscr{K}{critical} = 1.1$, where the natural frequencies $\omega_{i}:\, \mbb R_{\geq 0} \to [\subscr{\omega}{min},\subscr{\omega}{max}] = [0,1]$ are smooth, bounded, and distinct sinusoidal functions. Ultimately, each natural frequency $\omega_{i}(t)$ converges to $\omega_{i} + \sin(\pi t)$ with $\omega_{i} \in [0,1]$, and thus the relative acceleration $\ddot\Omega(t) = \ddot \omega(t) - \subscr{\ddot \omega}{avg}(t) \fvec 1_{n \times 1}$ converges to zero. The simulation illustrates the phase cohesiveness of the angles $\theta(t)$ in $\bar\Delta(\subscr{\gamma}{min})$, and the ultimate boundedness of the frequency variations (disagreement vector) $\dot \delta(t) = \dot{\theta}(t) - \subscr{\omega}{avg}(t) \fvec 1_{n \times 1}$ and their convergence to zero. The simulation further confirms the monotonicity of $V(\theta(t))$ in $\bar\Delta(\gamma)$ for $\gamma \in [\subscr{\gamma}{min},\subscr{\gamma}{max}]$. Ultimately, $V(\theta(t))$ converges to a constant value (strictly below $\subscr{\gamma}{min}$) as the frequencies converge.}
	\label{Fig: time-varying natural frequencies}
	}
\end{figure}

% -----------------------------------------------------------------------------------------------------%
\section{Synchronization of Multi-Rate Kuramoto Models}
\label{Section: Synchronization of Multi-Rate Kuramoto Models}

In this section we extend the results in Theorem \ref{Theorem: Necessary and Sufficient conditions} to the multi-rate Kuramoto model \eqref{eq: multi-rate Kuramoto model}. For the special case of second-order oscillators ($m=n$) with unit damping $D_{i}=1$ and uniform inertia $M_{i} = M > 0$, the literature \cite{YPC-SYH-SBH:11,HAT-AJL-SO:97,HAT-AJL-SO:97b,HH-GSJ-MYC:02,HH-MYC-JY-KSS:99,JAA-LLB-RS:00,JAA-LLB-CJPV-FR-RS:05} on the inertial effects on synchronization is controversial. Here we will rigorously prove that the inertial terms {\em do not affect} the location and local stability properties of equilibria of the multi-rate Kuramoto model \eqref{eq: multi-rate Kuramoto model}. In particular, the necessary and sufficient synchronization conditions as well as the synchronization frequency are independent\,of\,\,the\,inertiae $M_{i}$; they rather depend on the terms $D_{i}$ mimicking viscous damping and time constants.

% -----------------------------------------------------------------------------------------------------%
\subsection{A One-Parameter Family of Dynamical Systems and its Properties}
\label{Subsection: A Parameterized System and its Properties}
% -----------------------------------------------------------------------------------------------------%

In this subsection we will link the multi-rate Kuramoto model \eqref{eq: multi-rate Kuramoto model} and the first-order Kuramoto model \eqref{eq: Kuramoto model} through a parametrized system. Consider for $n_{1}, n_{2} \geq 0$ and $\lambda \in [0,1]$ the one-parameter family $\mc H_{\lambda}$ of dynamical systems combining dissipative Hamiltonian and gradient-like dynamics together with external forcing\,\,as
\begin{align}
	\mc H_{\lambda}:\;\;\;
	\begin{split}
	D_{1} \dot x_{1} 
	=&\; 
	F_{1} - \nabla_{1} H(x)
	\,, \\
	\begin{bmatrix}
	I_{n_{2}} & \fvec 0 \\  \fvec 0 & M
	\end{bmatrix}
	\begin{bmatrix}
	\dot x_{2}\\\dot x_{3}
	\end{bmatrix}
	=&\;
	\begin{bmatrix}
	\lambda D_{2}^{-1} F_{2} \\ (1-\lambda) F_{2}
	\end{bmatrix}
	+
	\\ &\;
	\left( 
	(1-\lambda)
	\begin{bmatrix}
	\fvec 0 & I_{n_{2}} \\ -I_{n_{2}} & \fvec 0
	\end{bmatrix}
	-
	\begin{bmatrix}
	\lambda D_{2}^{-1} & \fvec 0 \\ \fvec 0 & D_{2}
	\end{bmatrix}
	\right)
	\begin{bmatrix}
	\nabla_{2} H(x) \\  \nabla_{3} H(x)
	\end{bmatrix}
	,
	\end{split}
	\label{eq: Hlambda family}
\end{align}
where $x = (x_{1},x_{2},x_{3}) \in \mc X_{1} \times \mc X_{2} \times \mathbb R^{n_{2}} = \mc X$ is the state, and the sets $\mc X_{1}$ and $\mc X_{2}$ are smooth manifolds of dimensions $n_{1}$ and $n_{2}$, respectively. The matrices $D_{1} \in \mbb R^{n_{1} \times n_{1}}$, $D_{2} \in \mbb R^{n_{2} \times n_{2}}$ and $M \in \mbb R^{n_{2} \times n_{2}}$ are positive definite, $\fvec 0$ are zero matrices of appropriate dimension%
\footnote{We did not index the zero matrices $\fvec 0$ according to their dimension to avoid notational clutter.},
$F_{1} \in \mbb R^{n_{1}}$ and $F_{2} \in \mbb R^{n_{2}}$ are constant forcing terms, and $H:\, \mc X \to \mbb R$ is a smooth potential function with partial derivative $\nabla_{i}H(x) = \partial H(x)/\partial x_{i}$, gradient vector $\nabla H(x) = (\partial H(x) / \partial x)^{T} \in\mbb R^{n \times 1}$, and Hessian matrix $\nabla^{2} H(x) \in \mbb R^{n \times n}$. 

The parameterized system \eqref{eq: Hlambda family} continuously interpolates, as a function of $\lambda \in {[0,1]}$, between gradient-like and mixed dissipative Hamiltonian/gradient-like dynamics. For $\lambda = 1$, the system \eqref{eq: Hlambda family} reduces to gradient-like dynamics with forcing term $\fvec F  \!=\! [F_{1}^{T},F_{2}^{T},\fvec 0]^{T}$ and {\em time constant} (or system metric) $\fvec D \!=\! \mathrm{blkdiag}(D_{1},D_{2},D_{2}^{-1}M)$\,as
\begin{align}
	\mc H_{1}:\quad
	\fvec D \dot x =  \fvec F - \nabla H(x)
	\label{eq: Hlambda family - gradient}
	\,.
\end{align}
For $\lambda = 0$, the dynamics \eqref{eq: Hlambda family} reduce to gradient-like dynamics for $x_{1}$ and dissipative Hamiltonian (or Newtonian) dynamics for $(x_{2},x_{3})$ written as
\begin{align}
	\mc H_{0}:\quad
	\begin{split}
	D_{1} \dot x_{1} 
	&= 
	F_{1} - \nabla_{1} H(x)
	\,, \\
	\begin{bmatrix}
	I_{n_{2}} & \fvec 0 \\  \fvec 0 & M
	\end{bmatrix}
	\!
	\begin{bmatrix}
	\dot x_{2}\\\dot x_{3}
	\end{bmatrix}
	&=
	\begin{bmatrix}
	\fvec 0 \\ F_{2}
	\end{bmatrix}
	+
	\left(
	\begin{bmatrix}
	\fvec 0\!\!\!\! & \!I_{n_{2}} \\ -I_{n_{2}}\!\!\!\! & \!\fvec 0
	\end{bmatrix}
	-
	\begin{bmatrix}
	\fvec 0 & \fvec 0 \\ \fvec 0 & D_{2}
	\end{bmatrix}
	\right)
	\begin{bmatrix}
	\nabla_{2} H(x) \\  \nabla_{3} H(x)
	\end{bmatrix}
	\,.
	\end{split}
	\label{eq: Hlambda family - Hamiltonian}
\end{align}
It turns out that, independently of $\lambda \in {[0,1]}$, all parameterized systems of the form \eqref{eq: Hlambda family} have the same equilibria with the same local stability properties determined by potential function $H(x)$. The following theorem summarizes these facts.

\begin{theorem}[\bf Properties of $\mc H_{\lambda}$ family]
\label{Theorem: Properties of Hlambda family}
Consider for $\lambda \in [0,1]$ the one-parameter family $\mc H_{\lambda}$ of dynamical systems \eqref{eq: Hlambda family}. 
The following statements hold independent of the parameter $\lambda \in [0,1]$ and independent of the particular positive definite matrices $D_{1},D_{2},M$:
\smallskip
\begin{enumerate}

	\item {\bf Equilibria:} 	For all $\lambda \in [0,1]$ the equilibria of $\mc H_{\lambda}$ are given by the set $\mc E \triangleq \{x \in \mc X:\, \nabla H(x) = \fvec F \}$; and

	\item {\bf Local stability:} For any equilibrium $x^{*} \in \mc E$ and for all $\lambda \in [0,1]$, the inertia of the Jacobian of $\mc H_{\lambda}$ is given by the inertia of $-\nabla^{2}H(x^{*})$ and the corresponding center-eigenspace is given by the nullspace of $\nabla^{2}H(x^{*})$.
	
%	\item {\bf Global convergence on $\mbb R^{n_{1} + 2n_{2}}$:} Consider the dynamical system $\mc H_{\lambda}$ evolving on $\mbb R^{n_{1} + 2n_{2}}$ and the artificial potential function $\tilde H:\, R^{n_{1} + 2n_{2}} \to \mbb R$, $\tilde H(x) = - F_{1}^{T} x_{1} - F_{2}^{T}x_{2} + H(x_{1},x_{2},M^{1/2}x_{3})$. Then for any forward-complete solution $\map{x}{\real_{\geq0}}{\mc X}$ of $\mc H_{\lambda}$ the function $\tilde H(x(t))$ is non-increasing for all $\lambda \in [0,1]$. If the sublevel set $\Omega_{c} = \{ x \in \mc X:\, \tilde H(x) \leq c \}$ is compact, then for every initial condition $x(0) \in \Omega_c$ the corresponding solution $\map{x}{\real_{\geq0}}{\Omega_c}$ is forward-complete and converges to the set $\mc E \cap \Omega_c$, independently of $\lambda \in [0,1]$. \fdmargin{comments on ``evolving on $\mbb R^{n}$''}

\end{enumerate}
\end{theorem}
\smallskip

Statements 1) and 2) assert that normal hyperbolicity of the critical points of $H(x)$ can be directly related to local\,\,exponential (set) stability for any $\lambda\!\in\![0,1]$. This implies that  all vector fields $\mc H_{\lambda}$, $\lambda \in [0,1]$, are {\em locally topologically conjugate} \cite{CR:99}
%\cite[Section 4.7]{CR:99},\cite[Section 19.12]{SW:03} 
near a hyperbolic equilibrium point $x^{*} \in \mc E$.
In particular, near $x^{*} \in \mc E$, trajectories of the gradient vector field \eqref{eq: Hlambda family - gradient} can be continuously deformed\,\,to match trajectories of the Hamiltonian vector field \eqref{eq: Hlambda family - Hamiltonian} while preserving parameterization of time. This topological conjugacy holds also for hyperbolic equilibrium trajectories \cite[Theorem 6]{EACT-SEAM-PRCR:07} considered in synchronization. 
The similarity between second-order Hamiltonian systems and the corresponding first-order gradient flows is well-known in mechanical control systems \cite{DEK:88b,DEK:89}, in dynamic optimization \cite{FA:00,HA-PEM:10,XG-JM:09}, and in transient stability studies for power networks \cite{HDC-FFW:02,HDC-CCC:95,CCC:96}, but we are not aware of any result as general as Theorem \ref{Theorem: Properties of Hlambda family}. In \cite{HDC-FFW:02,HDC-CCC:95,CCC:96}, statements 1) and 2) are proved under the more stringent assumptions that $\mc H_{\lambda}$ has a finite number of isolated and hyperbolic equilibria.

\begin{remark}[\bf Extensions on Euclidean state spaces]
\normalfont
If the dynamical system $\mc H_{\lambda}$ is analyzed on the Euclidean space $\mbb R^{n_{1} + 2n_{2}}$, then it can be verified that the modified potential function $\tilde H:\, R^{n_{1} + 2n_{2}} \to \mbb R$, $\tilde H(x) = - F_{1}^{T} x_{1} - F_{2}^{T}x_{2} + H(x_{1},x_{2},M^{1/2}x_{3})$ is non-increasing along any forward-complete solution $\map{x}{\real_{\geq0}}{R^{n_{1} + 2n_{2}} }$ and for all $\lambda \in [0,1]$. Furthermore, if the sublevel set $\Omega_{c} = \{ x \in \mc X:\, \tilde H(x) \leq c \}$ is compact, then every solution initiating in $\Omega_c$ is bounded and forward-complete, and by the invariance principle \cite[Theorem 4.4]{HKK:02} it converges to the set $\mc E \cap \Omega_c$, independently of $\lambda \in [0,1]$. These statements can be refined under further structural assumptions on the potential function $\tilde H(x)$, and various other minimizing properties can be deduced, see \cite{FA:00,HA-PEM:10,XG-JM:09}. Additionally, if $\tilde H(x)$ constitutes an {\it energy function}, if all equilibria are hyperbolic, and if a one-parameter transversality condition is satisfied, then the separatrices of system \eqref{eq: Hlambda family} can be characterized accurately \cite{HDC-FFW:02,HDC-CCC:95,CCC:96}. For zero forcing $\fvec F = \fvec 0$, these convergence statements also hold on the possibly non-Euclidean space $\mc X$, and for non-zero forcing they hold locally on $\mc X$.
\oprocend
\end{remark}

\begin{proof}
To prove statement 1), we reformulate the parameterized dynamics \eqref{eq: Hlambda family}\,\,as
\begin{equation*}
	\begin{bmatrix}
		\dot x_{1} \\ \dot x_{2} \\ M \dot x_{3}
	\end{bmatrix}
	=
	\underbrace{
	\begin{bmatrix}
		D_{1}^{-1} & \fvec 0 & \fvec 0 \\
		\fvec 0 & \lambda D_{2}^{-1} & -(1-\lambda) I_{n_{2}} \\
		\fvec 0 & (1-\lambda) I_{n_{2}} & D_{2}
	\end{bmatrix}
	}_{\triangleq W_{\lambda}}
	\underbrace{
	\begin{bmatrix}
		F_{1} - \nabla_{1} H(x) \\ F_{2} - \nabla_{2} H(x) \\  - \nabla_{3} H(x)
	\end{bmatrix}
	}_{= \fvec F - \nabla H(x)}
	\,.
\end{equation*}
It follows from the {\it Schur determinant formula} \cite{FZ:05} that $\det(W_{\lambda}) = \det(D_{1}^{-1})(\lambda + (1-\lambda)^{2})$ is positive for all $\lambda \in {[0,1]}$. Hence, $W_{\lambda}$ is nonsingular for all $\lambda \in {[0,1]}$, and the equilibria of \eqref{eq: Hlambda family} are given by by the set $\mc E = \{x \in \mc X:\, \nabla H(x) = \fvec F \}$.
To prove statement 2) we analyze the Jacobian of $\mc H_{\lambda}$ at an equilibrium $x^{*} \in \mc E$ given by
\begin{equation}
	J_{\lambda}(x^{*})
	=
	\underbrace{
	\left[\begin{array}{cc|c}
		D_{1}^{-1}\!\!\! & 0\! & 0\! \\
		0\!\!\! & \lambda D_{2}^{-1}\! & (\lambda-1) M^{-1}\! \\ \hline
		0\!\!\! & (1-\lambda) M^{-1}\! & M^{-1}D_{2} M^{-1}\!
	\end{array}\right]
	}_{\triangleq S_{\lambda}}
	\;
	\underbrace{
	\left[\begin{array}{c|c}
	- I_{n_{1}+n_{2}} \! & \fvec 0 \\ \hline
	\fvec 0 \! & - M 
	\end{array}\right] \nabla^{2} H(x^{*})
	}_{\triangleq S(x^{*})}
	\label{eq: Jacobian of Hlambda}
	. 
\end{equation}
Again, we obtain $\det(S_{\lambda}) = \det(D_{1}^{-1})\det(D_{2}^{-1})\det(M^{-1}D_{2} M^{-1})(\lambda + (1-\lambda)^{2})$. Thus, $S_{\lambda}$ is nonsingular for $\lambda \in {[0,1]}$, and the nullspace of $J_{\lambda}(x^{*})$ is given by $\mathrm{ker}\nabla^{2}H(x^{*})$ (independently of $\lambda \in [0,1]$).
To show that the stability properties of the equilibrium $x^{*} \in \mc E$ are independent of $\lambda \in [0,1]$, we prove that the inertia of the Jacobian $J_{\lambda}(x^{*})$ depends only on $S(x^{*})$ and not on $\lambda \in [0,1]$. For the invariance of the inertia we appeal to the {\it main inertia theorem for positive semi-definite matrices} \cite[Theorem 5]{DC-HS:62}. Note that $J_{\lambda}(x^{*})$ and $J_{\lambda}(x^{*})^{T}$ have the same eigenvalues. % and $S(x^{*})=S(x^{*})^{T}$. 
Let $A \triangleq J_{\lambda}(x^{*})^{T}$ and $P \triangleq S(x^{*})$, and consider the matrix $Q$ defined via the Lyapunov equality as
\begin{equation*}
	Q 
	\triangleq
	\frac{1}{2} \left( AP + PA^{T} \right) 
%	=  
%	\frac{1}{2} \left( P S_{\lambda}^{T}P + P S_{\lambda}P \right)
	=
	P 
	\begin{bmatrix}
		D_{1}^{-1} & 0 & 0\\
		0 & \lambda D_{2}^{-1} & 0 \\
		0 & 0 & M^{-1}D_{2} M^{-1}
	\end{bmatrix} 
	P
	\,.
\end{equation*} 
Note that $Q$ is positive semidefinite for $\lambda \geq 0$, and for $\lambda \neq 0$ the nullspaces of $Q$ and $P$ coincide, i.e., $\mathrm{ker}Q = \mathrm{ker}P$. Hence, for $\lambda \in {]0,1]}$ the assumptions of \cite[Theorem 5]{DC-HS:62} are satisfied, and it follows that the non-zero inertia of $A=J_{\lambda}(x^{*})^{T}$ (restricted to image of $A$) corresponds to the non-zero inertia of $P$. Hence, the non-zero inertia of $J_{\lambda}(x^{*})$ is {\it independent} of $\lambda \in {]0,1]}$, and possible zero eigenvalues correspond to $\mathrm{ker}J_{\lambda}(x^{*}) = \mathrm{ker}\nabla^{2}H(x^{*})$. % (also independent of $\lambda \in {]0,1]}$). 
To handle the case $\lambda = 0$ we invoke continuity arguments. Since the eigenvalues of $J_{\lambda}(x^{*})$ are continuous functions of the matrix elements, the inertia of $J_{0}(x^{*})$ is the same as the inertia of $J_{\lambda}(x^{*})$ for $\lambda>0$ sufficiently small. Since the inertia of $J_{\lambda}(x^{*})$, $\lambda \in {]0,1]}$, equals the inertia of $P$ (which is independent of $\lambda$), it follows that the inertia of $J_{\lambda}(x^{*})$ equal the inertia of $P$ for all $\lambda \in [0,1]$. 

Finally, since $\mathrm{blkdiag}(I_{n_{1}+n_{2}},M)$ is positive definite, {\it Sylvester's inertia theorem} \cite{DC-HS:62} asserts that the inertia of $P = \mathrm{blkdiag}(I_{n_{1}+n_{2}},M) (-\nabla^{2} H(x^{*}))$ equals the inertia of  $-\nabla^{2} H(x^{*})$. In conclusion, the inertia and the nullspace of  $J_{\lambda}(x^{*})$ equal the inertia of $-\nabla^{2} H(x^{*})$ and $\mathrm{ker}\nabla^{2}H(x^{*})$. This completes the proof of Theorem \ref{Theorem: Properties of Hlambda family}.
\end{proof}

% -----------------------------------------------------------------------------------------------------%
\subsection{Equivalence of Local Synchronization Conditons}
\label{Subsection: Equivalence of Local Synchronization Conditions}
% -----------------------------------------------------------------------------------------------------%

As a consequence of Theorem \ref{Theorem: Properties of Hlambda family}, we can link synchronization in the multi-rate Kuramoto model \eqref{eq: multi-rate Kuramoto model} and in the regular Kuramoto model \eqref{eq: Kuramoto model}. Since Theorem \ref{Theorem: Properties of Hlambda family} is valid only for equilibria, we convert synchronization to stability of an equilibrium manifold by changing coordinates to a rotating frame. The explicit synchronization frequency $\subscr{\omega}{sync} \in \mbb R$\,\,of the multi-rate Kuramoto model \eqref{eq: multi-rate Kuramoto model} is obtained by summing over all equations \eqref{eq: multi-rate Kuramoto model}\,as
\begin{equation}
	\sum\nolimits_{i=1}^{m} M_{i} \ddot \theta_{i} + \sum\nolimits_{i=1}^{n} D_{i} \dot \theta_{i} = \sum\nolimits_{i=1}^{n} \omega_{i}
	\label{eq: sum over all dynamics}
	\,.
\end{equation}
In the frequency-synchronized case when all $\ddot \theta_{i} = 0$ and $\dot \theta_{i} = \subscr{\omega}{sync}$, equation \eqref{eq: sum over all dynamics} simplifies to $\sum_{i=1}^{n} D_{i} \subscr{\omega}{sync} = \sum_{i=1}^{n} \omega_{i}$. We conclude that the synchronization frequency of the multi-rate Kuramoto model is given by $\subscr{\omega}{sync} \triangleq \sum_{i=1}^{n} \omega_{i}/\sum_{i=1}^{n} D_{i}$. Accordingly, define the {\it first-order multi-rate Kuramoto model} by dropping the inertia term\,as
\begin{equation}
	D_{i} \dot \theta_{i}
	=
	\omega_{i} - \frac{K}{n} \sum\nolimits_{j=1}^{n} \sin(\theta_{i} - \theta_{j})
	\,,\quad i \in \until n\,,
	\label{eq: 1st order multi-rate Kuramoto model}
\end{equation}
and the globally exponentially stable {\it frequency dynamics} as
\begin{equation}
	\frac{d}{dt}\, \dot\theta_{i}
	=
	-M_{i}^{-1} D_{i}\, \bigl( \dot \theta_{i} - \subscr{\omega}{sync} \bigr)
	\,, \quad i \in \{1 , \dots, m\}\,,
	\label{eq: frequency dynamics}
\end{equation}
where $M_{i}$, $D_{i}$, $\omega_{i}$, and $K$ take the same values as the corresponding parameters for the multi-rate Kuramoto model \eqref{eq: multi-rate Kuramoto model}. It can be verified that the multi-rate Kuramoto model \eqref{eq: multi-rate Kuramoto model} and its first-order variant \eqref{eq: 1st order multi-rate Kuramoto model} have the same synchronization frequency. 

Finally, let $\bar \omega_{i} \triangleq \omega_{i} - D_{i} \subscr{\omega}{sync}$ and define the {\it scaled Kuramoto model} by
\begin{equation}
	\dot \theta_{i}
	=
	\bar \omega_{i} - \frac{K}{n} \sum\nolimits_{j=1}^{n} \sin(\theta_{i} - \theta_{j})
	\,,\quad i \in \until n\,,
	\label{eq: scaled Kuramoto model}
\end{equation}
and its associated {\it scaled frequency dynamics} by
\begin{equation}
	\frac{d}{dt}\, \dot\theta_{i}
	=
	-M_{i}^{-1} D_{i}\, \dot \theta_{i}
	\,, \quad i \in \{1 , \dots, m\}\,.
	\label{eq: scaled frequency dynamics}
\end{equation}
The scaled model \eqref{eq: scaled Kuramoto model}-\eqref{eq: scaled frequency dynamics} corresponds to the dynamics \eqref{eq: 1st order multi-rate Kuramoto model}-\eqref{eq: frequency dynamics} formulated\,in\,a rotating frame with frequency $\subscr{\omega}{sync}$ and after unitizing the time constants $D_{i}$ in \eqref{eq: 1st order multi-rate Kuramoto model}.

Notice that the multi-rate Kuramoto model \eqref{eq: multi-rate Kuramoto model}, its first-order variant \eqref{eq: 1st order multi-rate Kuramoto model} together with frequency dynamics \eqref{eq: frequency dynamics} (formulated in a rotating frame with frequency $\subscr{\omega}{sync}$), and the scaled Kuramoto model \eqref{eq: scaled Kuramoto model} together with scaled frequency dynamics \eqref{eq: scaled frequency dynamics} are instances of the parameterized system \eqref{eq: Hlambda family} with the forcing terms $\omega_{i}$ and the potential $H:\, \mbb T^{n} \times \mbb R^{m} \to \mbb R$, $H(\theta,\dot \theta) = \frac{1}{2} \dot\theta^{T} \dot\theta - \frac{K}{n} \sum_{i,j=1}^{n} \cos(\theta_{i} - \theta_{j})$ defined up to a constant value. In the sequel, we seek to apply Theorem\,\ref{Theorem: Properties of Hlambda family} to these three models.

 For a rigorous reasoning, we define a two-parameter family of functions $\phi_{r,s}:\,  \mbb R_{\geq 0} \to \mbb T$ of the form $\phi_{r,s}(t) \triangleq r + s \cdot t \pmod{2\pi}$, where $r \in \mbb T$ and $s \in \mbb R$. Consider for $(r_{1},\dots,r_{n}) \in \bar\Delta(\gamma)$, $\gamma \in {[0,\pi[}$ the composite function 
\begin{equation}
	\Phi_{\gamma,s}:\,  \mbb R_{\geq 0} \to \mbb T^{n}
	\,,\quad
	\Phi_{\gamma,s}(t) \triangleq \bigl( \phi_{r_{1},s}(t), \dots, \phi_{r_{n},s}(t) \bigr)
	\label{eq: composite synchronized trajectory}
\end{equation}
mimicking synchronized trajectories of the three Kuramoto models \eqref{eq: multi-rate Kuramoto model}, \eqref{eq: 1st order multi-rate Kuramoto model}, and \eqref{eq: scaled Kuramoto model}. We now have all ingredients to state the following result on synchronization.

\begin{theorem}
{\bf (Synchronization Equivalence)}
\label{Theorem: Local Equivalence of First and Second Order Synchronization}
Consider the multi-rate Kuramoto model \eqref{eq: multi-rate Kuramoto model}, its first-order variant \eqref{eq: 1st order multi-rate Kuramoto model}, and the scaled Kuramoto model \eqref{eq: scaled Kuramoto model} with $\bar \omega_{i} = \omega_{i} - D_{i} \subscr{\omega}{sync}$, where $\subscr{\omega}{sync} = \sum_{k=1}^{n} \omega_{k} / \sum_{k=1}^{n} D_{k}$. The following statements are equivalent for any $\gamma \in {[0,\pi[}$, $t \geq 0$, and any function $\Phi_{\gamma,\subscr{\omega}{sync}}(t)$ defined in \eqref{eq: composite synchronized trajectory}:
%\smallskip
\begin{enumerate}

	\item[(i)] $(\Phi_{\gamma,\subscr{\omega}{sync}}(t), \subscr{\omega}{sync} \fvec 1_{m \times 1})$ parametrizes a locally exponentially stable synchronized trajectory $(\theta(t),\dot \theta(t))$ of the multi-rate Kuramoto model \eqref{eq: multi-rate Kuramoto model};
	
	\item[(ii)] $\Phi_{\gamma,\subscr{\omega}{sync}}(t)$ parametrizes a locally exponentially stable synchronized trajectory $\theta(t)$ of the first-order multi-rate Kuramoto model \eqref{eq: 1st order multi-rate Kuramoto model}; and
	
	\item[(iii)] $\Phi_{\gamma,0}(t)$ parametrizes a locally exponentially stable synchronized equilibrium trajectory $\theta(t)$ of the scaled Kuramoto model \eqref{eq: scaled Kuramoto model}.
		
\end{enumerate}
If the equivalent statements (i), (ii), and (iii) are true, then, locally near their\,respective synchronization manifolds, the multi-rate Kuramoto model \eqref{eq: multi-rate Kuramoto model}, its first-order variant \eqref{eq: 1st order multi-rate Kuramoto model} together with the frequency dynamics \eqref{eq: frequency dynamics}, and the scaled Kuramoto model \eqref{eq: scaled Kuramoto model} together with the scaled frequency dynamics \eqref{eq: scaled frequency dynamics} are topologically conjugate.
\end{theorem}
\smallskip

For purely second-order Kuramoto oscillators \eqref{eq: multi-rate Kuramoto model} (with $n=m$), Theorem \ref{Theorem: Properties of Hlambda family} and Theorem \ref{Theorem: Local Equivalence of First and Second Order Synchronization}  essentially state that the locations and stability properties of the {\em foci} of second-order Kuramoto oscillators (with damped oscillatory dynamics) are equivalent to those of the {\em nodes} of the scaled Kuramoto model \eqref{eq: scaled Kuramoto model} and the scaled frequency dynamics \eqref{eq: scaled frequency dynamics} (with overdamped dynamics), as illustrated in Figure\,\,\ref{Fig: 1st and 2nd order simulation}.
\begin{figure}[htbp]
	\centering{
	\includegraphics[scale=0.687]{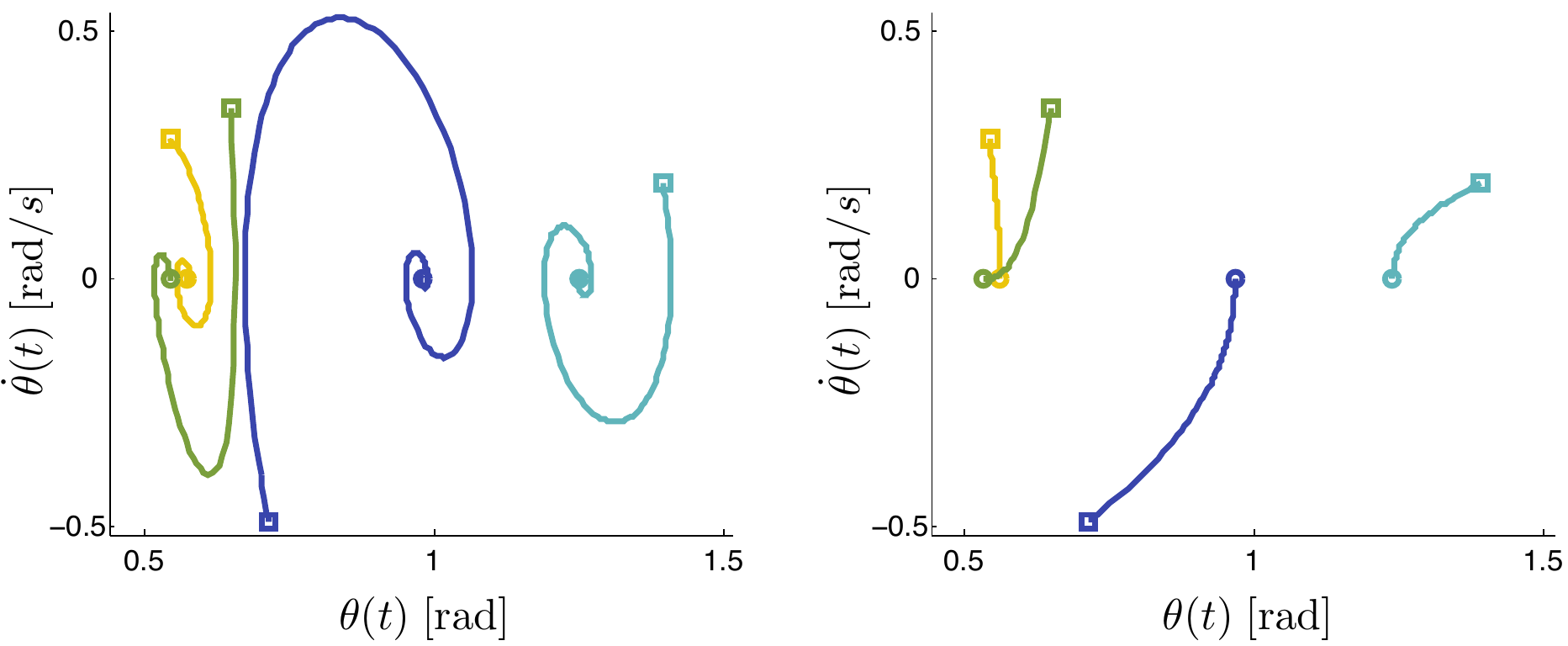}
	\caption{Phase space plot of a network of $n=4$ second-order Kuramoto oscillators \eqref{eq: multi-rate Kuramoto model} with $n=m$ (left plot) and the corresponding first-order scaled Kuramoto oscillators \eqref{eq: scaled Kuramoto model} together with the scaled frequency dynamics \eqref{eq: scaled frequency dynamics} (right plot). The natural frequencies $\omega_{i}$, damping terms $D_{i}$, and coupling strength $K$ are such that $\subscr{\omega}{sync} = 0$ and $K/\subscr{K}{critical} = 1.1$. From the same initial configuration $\theta(0)$ (denoted by $\blacksquare$) both first and second-order oscillators converge exponentially to the same nearby phase-locked equilibria  (denoted by {\large$\bf\bullet$}) as predicted by Theorems \ref{Theorem: Properties of Hlambda family} and \ref{Theorem: Local Equivalence of First and Second Order Synchronization}.}
	\label{Fig: 1st and 2nd order simulation}
	}
\end{figure}

\begin{proof}
By Definition, a synchronized trajectory of the multi-rate Kuramoto model \eqref{eq: multi-rate Kuramoto model} is of the form $(\theta(t),\dot \theta(t)) \in (\Phi_{\gamma,\subscr{\omega}{sync}}(t),\subscr{\omega}{sync} \fvec 1_{m \times 1})$ for $\gamma \in {[0,\pi[}$ and $t \geq 0$. In a rotating frame with frequency $\subscr{\omega}{sync}$, the multi-rate Kuramoto model \eqref{eq: multi-rate Kuramoto model} reads as 
\begin{align}
	\begin{split}
	M \ddot \theta_{i} + D_{i} \dot \theta_{i}
	&=
	\bar \omega_{i} - \frac{K}{n} \sum\nolimits_{j=1}^{n} \sin(\theta_{i} - \theta_{j})
	\,, \quad i \in \until m \,,
	\\
	D_{i} \dot \theta_{i}
	&=
	\bar \omega_{i} - \frac{K}{n} \sum\nolimits_{j=1}^{n} \sin(\theta_{i} - \theta_{j})
	\,, \quad i \in \{m+1,\dots,n\} \,.
	\end{split}
	\label{eq: multi-rate Kuramoto model - rotating frame}
\end{align}
Hence, an exponentially synchronized trajectory of \eqref{eq: multi-rate Kuramoto model - rotating frame} is an equilibrium solution determined up to a translational invariance in $\mbb S^{1}$ and satisfies $(\theta(t),\dot \theta(t)) \in (\Phi_{\gamma,0}(t),\fvec 0_{m \times 1})$. Hence, the phase-synchronized {\it orbit} $(\Phi_{\gamma,0}(t),\fvec 0_{m \times 1})$, understood as a geometric object in $\mbb T^{n} \times \mbb R^{m}$, constitutes a one-dimensional equilibrium manifold of the multi-rate Kuramoto model \eqref{eq: multi-rate Kuramoto model - rotating frame}. 
After factoring out the translational invariance of the angular variable $\theta$, the exponentially-synchronized orbit $(\Phi_{\gamma,0}(t),\fvec 0_{m \times 1})$ corresponds to an isolated equilibrium of \eqref{eq: multi-rate Kuramoto model - rotating frame} in the quotient space $\mbb T^{n} \setminus \mbb S^{1} \times \mbb R^{m}$. Since an isolated equilibrium of a smooth nonlinear system with bounded and Lipschitz Jacobian is exponentially stable if and only if the Jacobian is a Hurwitz matrix \cite[Theorem 4.15]{HKK:02}, the locally exponentially stable orbit $(\Phi_{\gamma,0}(t),\fvec 0_{m \times 1})$ must be hyperbolic in the quotient space $\mbb T^{n} \setminus \mbb S^{1} \times \mbb R^{m}$. Therefore, the equilibrium trajectory $(\Phi_{\gamma,0}(t),\fvec 0_{m \times 1})$ is exponentially stable in $\mbb T^{n} \times \mbb R^{m}$ if and only if the Jacobian of \eqref{eq: multi-rate Kuramoto model - rotating frame} evaluated along $(\Phi_{\gamma,0}(t),\fvec 0_{m \times 1})$, has $n+m-1$ stable eigenvalues and one zero eigenvalue corresponding to the translational invariance in $\mbb S^{1}$. 

By an analogous reasoning we reach the same conclusion for the first-order multi-rate Kuramoto model \eqref{eq: 1st order multi-rate Kuramoto model} (formulated in a rotating frame with frequency $\subscr{\omega}{sync}$) and for the scaled Kuramoto model \eqref{eq: scaled Kuramoto model}: the exponentially-synchronized trajectory $\Phi_{\gamma,0}(t) \in \mbb T^{n}$ is exponentially stable if and only if the Jacobian of \eqref{eq: scaled Kuramoto model} evaluated along $\Phi_{\gamma,0}(t)$ has $n-1$ stable eigenvalues and one zero eigenvalue.
Finally, recall that  the multi-rate Kuramoto model \eqref{eq: multi-rate Kuramoto model - rotating frame}, its first-order variant \eqref{eq: 1st order multi-rate Kuramoto model} together with frequency dynamics \eqref{eq: frequency dynamics} (in a rotating frame), and the scaled Kuramoto model \eqref{eq: scaled Kuramoto model} together with scaled frequency dynamics \eqref{eq: scaled frequency dynamics} are all instances of the parameterized system \eqref{eq: Hlambda family}. Therefore, by Theorem \ref{Theorem: Properties of Hlambda family}, the corresponding Jacobians have the same inertia and local exponential stability of one system implies local exponential stability of the other system. This concludes the proof of the equivalences (i) $\Leftrightarrow$ (ii) $\Leftrightarrow$ (iii).

We now prove the final conjugacy statement. By the {\em generalized Hartman-Grobman theorem} \cite[Theorem 6]{EACT-SEAM-PRCR:07}, the trajectories of the three vector fields \eqref{eq: multi-rate Kuramoto model - rotating frame}, \eqref{eq: 1st order multi-rate Kuramoto model}-\eqref{eq: frequency dynamics} (formulated in a rotating frame), and \eqref{eq: scaled Kuramoto model}-\eqref{eq: scaled frequency dynamics} are locally topologically conjugate to the flow generated by their respective linearized vector fields (locally near $(\Phi_{\gamma,0}(t), \fvec 0_{m \times 1})$. Since the three vector fields \eqref{eq: multi-rate Kuramoto model - rotating frame}, \eqref{eq: 1st order multi-rate Kuramoto model}-\eqref{eq: frequency dynamics}, and \eqref{eq: scaled Kuramoto model}-\eqref{eq: scaled frequency dynamics} are hyperbolic with respect to $(\Phi_{\gamma,0}(t), \fvec 0_{m \times 1})$ and their respective Jacobians have the same hyperbolic inertia (besides the common one-dimensional center eigenspace corresponding to $(\Phi_{\gamma,0}(t), \fvec 0_{m \times 1})$, the corresponding three linearized dynamics are topologically conjugate \cite[Theorem 7.1]{CR:99}. In summary, the trajectories generated by the three vector fields \eqref{eq: multi-rate Kuramoto model - rotating frame}, \eqref{eq: 1st order multi-rate Kuramoto model}-\eqref{eq: frequency dynamics} (formulated in a rotating frame), and \eqref{eq: scaled Kuramoto model}-\eqref{eq: scaled frequency dynamics} are locally topologically conjugate near the equilibrium manifold $(\Phi_{\gamma,0}(t), \fvec 0_{m \times 1})$.
\end{proof}

\begin{remark}[\bf Alternative ways from first to second-order Kuramoto models]
\normalfont
Alternative methods to relate stability properties from the first-order Kuramoto model \eqref{eq: Kuramoto model} to the multi-rate model \eqref{eq: multi-rate Kuramoto model} include second-order Gronwall's inequalities \cite{YPC-SYH-SBH:11}, strict Lyapunov functions for mechanical systems \cite{DEK:88b,DEK:89}, and singular perturbation analysis \cite{FD-FB:09z}. 
It should be noted that the approaches \cite{YPC-SYH-SBH:11,DEK:88b,DEK:89} are limited to purely second-order systems, the second-order Gronwall inequality approach \cite{YPC-SYH-SBH:11} has been carried out only for uniform inertia $M_{i} = M$ and unit damping $D_{i} = 1$, and the Lyapunov approach \cite{DEK:88b,DEK:89} is limited to potential-based Lyapunov functions and seems not extendable to our contraction-based Lyapunov function used in the proof of Theorem \ref{Theorem: Necessary and Sufficient conditions}. Finally, the singular perturbation approach \cite{FD-FB:09z} requires a sufficiently small inertia over damping ratio $\epsilon \triangleq \max_{i \in \until m}\{M_{i}/D_{i}\}$. 

As compared with these alternative methods, Theorem \ref{Theorem: Local Equivalence of First and Second Order Synchronization} applies to the multi-rate Kuramoto model \eqref{eq: multi-rate Kuramoto model} with mixed first and second-order dynamics, for all values of $M_{i}>0$ and $D_{i}>0$, and without additional assumptions. Finally, it is instructive to note that the first-order multi-rate Kuramoto\,dynamics \eqref{eq: 1st order multi-rate Kuramoto model} and the  frequency dynamics \eqref{eq: frequency dynamics} (in the time-scale $t/\epsilon$) correspond to the reduced slow system and\,the\,\,fast boundary layer model in the singular perturbation approach \cite[Theorem IV.2]{FD-FB:09z}. 
\oprocend
\end{remark}

% -----------------------------------------------------------------------------------------------------%
\subsection{Synchronization in the Multi-Rate Kuramoto Model}
\label{Subsection: Synchronization in the multi-rate Kuramoto model}
% -----------------------------------------------------------------------------------------------------%

Theorems \ref{Theorem: Properties of Hlambda family} and \ref{Theorem: Local Equivalence of First and Second Order Synchronization} together with Theorem \ref{Theorem: Necessary and Sufficient conditions} on the first-order Kuramoto model \eqref{eq: Kuramoto model} allow us to state our final conditions on synchronization in the multi-rate Kuramoto model\,\eqref{eq: multi-rate Kuramoto model}.

\begin{theorem}{\bf(Exponential Synchronization in the Multi-Rate Kuramoto Model)}
\label{Theorem: Exponential Synchronization in the Multi-Rate Kuramoto Model}
Consider the multi-rate Kuramoto model \eqref{eq: multi-rate Kuramoto model} and let $\bar \omega_{i} = \omega_{i} - D_{i} \subscr{\omega}{sync}$, where $\subscr{\omega}{sync} = \sum_{k=1}^{n} \omega_{k} / \sum_{k=1}^{n} D_{k}$. Then the following statements hold:

\smallskip
{\bf 1) Exponential synchronization:} The following two statements are equivalent:
\begin{enumerate}

	\item[(i)] the coupling strength $K$ is larger than the maximum non-uniformity among the scaled natural frequencies, i.e., $K > \subscr{K}{critical} \triangleq  \subscr{\bar\omega}{max} - \subscr{\bar\omega}{min}$; and

	\item[(ii)] there exists an arc length $\subscr{\gamma}{min} \in {[0,\pi/2[}$, such that the multi-rate Kuramoto model \eqref{eq: multi-rate Kuramoto model} has a locally exponentially stable synchronized solution with synchronization frequency $\subscr{\omega}{sync}$ and phase cohesive in $\bar\Delta(\subscr{\gamma}{min})$
for all $n \geq 2$, for all $m \in {[0,n]}$, for all inertiae $M_{j} > 0$, $j \in \until m$, and for all possible $D_{i} > 0$ and $\omega_{i} \in \mbb R$ satisfying $\bar \omega_{i} = \omega_{i} - D_{i} \subscr{\omega}{sync} \in {[\subscr{\bar\omega}{max} , \subscr{\bar\omega}{min}]}$, $i \in \until n$.

\end{enumerate}
Moreover, in either of the two equivalent cases (i) and (ii), the ratio $\subscr{K}{critical}/K$  and the arc length $\subscr{\gamma}{min} \in {[0,\pi/2[}$ are related uniquely via ${\subscr{K}{critical}}/K = \sin(\subscr{\gamma}{min})$.

\smallskip
{\bf 2) Phase synchronization:} The following two statements are equivalent:
\begin{enumerate}
	
	\item[(iii)] there exists a constant $\bar s \in \mbb R$ such that $\omega_{i} = D_{i} \bar s$ for all $i \in \until n$; and
		
	\item[(iv)] there exists an almost globally exponentially stable phase-synchronized solution with constant synchronization frequency $\subscr{\bar \omega}{sync} \in \mathbb R$.

\end{enumerate}
Moreover, in either of the two equivalent cases (iii) and (iv), the constant $\bar s$ and the synchronization frequency $\subscr{\bar \omega}{sync}$ are related uniquely via $\bar s \equiv \subscr{\bar \omega}{sync}$, and the the asymptotic synchronization phase is given by $\sum_{i=1}^{n} D_{i} \theta_{i}(0)/ \sum_{i=1}^{n} D_{i}  + \subscr{\bar \omega}{sync} t \pmod{2\pi}$.
\end{theorem}
\smallskip

The following remarks concerning Theorem \ref{Theorem: Exponential Synchronization in the Multi-Rate Kuramoto Model} are in order. 
First, notice that Theorem \ref{Theorem: Exponential Synchronization in the Multi-Rate Kuramoto Model} is in perfect agreement with the results derived in \cite{HH-MYC-JY-KSS:99} for the case of two second-order Kuramoto oscillators.
Second, Theorem \ref{Theorem: Exponential Synchronization in the Multi-Rate Kuramoto Model} shows that phase synchronization is independent of the inertial coefficients $M_{i}$, thereby improving the sufficient conditions presented in \cite[Theorems 4.1 and 4.2]{YPC-SYH-SBH:11} and confirming the results in \cite{JAA-LLB-CJPV-FR-RS:05,JAA-LLB-RS:00} derived for the infinite-dimensional case. Furthermore, phase synchronization occurs almost globally which improves the region of attraction presented in \cite{YPC-SYH-SBH:11} and naturally generalizes the result known for the first-order model \cite[Theorem 1]{RS-DP-NEL:07}.
Third, as in Section \ref{Section: New, tight, and explicit bound}, the bound on $\subscr{K}{critical}$ presented in (i) is only tight and may be conservative for a particular set of natural frequencies. Since the multi-rate Kuramoto model \eqref{eq: multi-rate Kuramoto model} is an instance of the parameterized system considered in Theorem \ref{Theorem: Properties of Hlambda family}, it has the same equilibria and the same stability properties as the scaled Kuramoto model \eqref{eq: scaled Kuramoto model} (together with the frequency dynamics \eqref{eq: scaled frequency dynamics}). Hence, the implicit formulae \eqref{eq: Verwoerd bound 1}-\eqref{eq: Verwoerd bound 2} can be applied to the scaled Kuramoto model \eqref{eq: scaled Kuramoto model} to find the exact critical coupling for a given set of natural frequencies. 
Fourth, we remark that every local bifurcation in the multi-rate Kuramoto model \eqref{eq: multi-rate Kuramoto model} is independent of the inertiae $M_{i}$ since it can be analyzed locally by means of the scaled Kuramoto model \eqref{eq: scaled Kuramoto model}. Moreover, the asymptotic magnitude of the order parameter determined by the location of phase-locked equilibria is also independent of the inertiae.
Fifth and finally, Theorems \ref{Theorem: Properties of Hlambda family} and \ref{Theorem: Local Equivalence of First and Second Order Synchronization} apply to any variant of the multi-rate Kuramoto model \eqref{eq: multi-rate Kuramoto model} that can be written in the forced Hamiltonian and gradient form \eqref{eq: Hlambda family} with normally hyperbolic equilibria. For example, the results on almost global phase synchronization for a connected and undirected coupling topology \cite[Proposition 3.3.2]{AS:09} and for state-dependent coupling weights \cite{LS:10} can be directly applied to the multi-rate Kuramoto model \eqref{eq: multi-rate Kuramoto model}.

Based on the results in this section, we conclude that the inertial terms do not affect the location and local stability properties of synchronized trajectories in the multi-rate Kuramoto model \eqref{eq: multi-rate Kuramoto model}. However, the inertiae may still affect the transient synchronization behavior, for example, the convergence rates, the shape of separatrices and basins of attractions, and the qualitative (possibly oscillatory) transient dynamics.

\begin{proof}
We begin by proving the equivalence (i) $\Leftrightarrow$ (ii). By Theorem \ref{Theorem: Local Equivalence of First and Second Order Synchronization}, a locally exponentially stable synchronized trajectory of the multi-rate Kuramoto model \eqref{eq: multi-rate Kuramoto model} exists if and only if there exists a locally exponentially stable equilibrium of the corresponding scaled Kuramoto model \eqref{eq: scaled Kuramoto model}. The latter is true if and only if statement (i) holds, see Theorem \ref{Theorem: Necessary and Sufficient conditions}. Moreover, Theorem \ref{Theorem: Necessary and Sufficient conditions} asserts that a synchronized solution is phase cohesive in $\bar\Delta(\subscr{\gamma}{min})$. This proves the equivalence (i) $\Leftrightarrow$ (ii).

We next prove the implication (iv) $\!\implies\!$ (iii). By assumption, there exist constants $\subscr{\theta}{sync} \in \mbb T$ and $\subscr{\bar \omega}{sync} \in \mbb R$ such that $\theta_{i}(t) = \subscr{\theta}{sync} + \subscr{\bar \omega}{sync}t \pmod{2\pi}$, $\dot \theta_{i}(t) = \subscr{\bar\omega}{sync}$, and $\ddot \theta_{i}(t) = 0$, $i \in \until n$. In the phase-synchronized case, the dynamics \eqref{eq: multi-rate Kuramoto model} then read as $D_{i} \subscr{\bar \omega}{sync} = \omega_{i}$ for all $i \in \until n$. Hence, a necessary condition for the existence of phase-synchronized solutions is that all ratios $\omega_{i}/D_{i} = \subscr{\bar \omega}{sync}$ are constant. 

In order to prove the converse implication (iii) $\implies$ (iv), let $\bar s= \subscr{\bar \omega}{sync}$ and consider the model \eqref{eq: multi-rate Kuramoto model} written in a rotating frame with frequency $\subscr{\bar\omega}{sync}$ as
\begin{align}
	\begin{split}
	M \ddot \theta_{i} + D_{i} \dot \theta_{i}
	&= 
	- \frac{K}{n} \sum_{i=1}^{n} \sin(\theta_{i}-\theta_{j})
	\,, \quad i \in \until m \,,
	\\
	D_{i} \dot \theta_{i}
	&=
	- \frac{K}{n} \sum_{i=1}^{n} \sin(\theta_{i}-\theta_{j})
	\,, \quad i \in \{m+1,\dots,n\} \,.
	\end{split}
	\label{eq: multi-rate Kuramoto model - rotating frame - phase sync}
\end{align}
Note that \eqref{eq: multi-rate Kuramoto model - rotating frame - phase sync} is an unforced and dissipative Hamiltonian system, and the corresponding energy function $V(\theta,\dot \theta) = \frac{1}{2} \dot\theta^{T} M\dot\theta - \frac{K}{n} \sum_{i,j=1}^{n} \cos(\theta_{i} - \theta_{j})$ is non-increasing along trajectories. Since the sublevel sets of $V(\theta,\dot \theta)$ are compact, the invariance principle \cite[Theorem 4.4]{HKK:02} implies that every solution converges to set of equilibria. By Theorem \ref{Theorem: Local Equivalence of First and Second Order Synchronization}, we conclude that the phase-synchronized equilibrium of \eqref{eq: multi-rate Kuramoto model - rotating frame - phase sync} is locally exponentially stable if and only if the the phase-synchronized equilibrium of the corresponding scaled Kuramoto model \eqref{eq: scaled Kuramoto model} with $\bar \omega_{i} = 0$ is exponentially stable. By \cite[Theorem 1]{RS-DP-NEL:07}, the latter statement is true, all other equilibria are locally unstable, and thus the region of attraction is almost global. This concludes the proof of (iii)\,$\Leftrightarrow$\,(iv).

To obtain the explicit synchronization phase, we sum over all equations \eqref{eq: multi-rate Kuramoto model - rotating frame - phase sync} to obtain
$
\sum\nolimits_{i=1}^{m} M_{i} \ddot \theta_{i} + \sum\nolimits_{i=1}^{n} D_{i} \dot \theta_{i} = 0
$.
Integration of this equation along phase-synchronized solutions yields that $\sum_{i=1}^{n} D_{i} \theta_{i}(t) = \sum_{i=1}^{n} D_{i} \theta_{i}(0)$ is constant for all $t \geq 0$, where we already accounted for $\dot \theta_{i}(t) = 0$ for all $i \in \until n$ and all $t \geq 0$.
Hence, the synchronization phase is given by a weighted average of the initial conditions $\sum_{i=1}^{n} D_{i} \theta_{i}(0)/ \sum_{i=1}^{n} D_{i}$. In the original coordinates (non-rotating frame) the synchronization phase is then given by $\sum_{i=1}^{n} D_{i} \theta_{i}(0)/ \sum_{i=1}^{n} D_{i} + \subscr{\dot \theta}{sync} t$.
\end{proof}

% -----------------------------------------------------------------------------------------------------%
\section{Conclusions}
\label{Section: Conclusions}
% -----------------------------------------------------------------------------------------------------%

This paper reviewed various bounds on the critical coupling strength in the Kuramoto model, formally introduced the powerful concept of phase cohesiveness, and presented an explicit and tight bound sufficient for synchronization in the Kuramoto model. This bound is necessary and sufficient for arbitrary distributions of the natural frequencies and tight for the particular case, where only implicit bounds are known. Furthermore, a general practical stability result as well as various performance measures have been derived as a function of the multiplicative gap in the bound. Finally, we partially extended these results to the multi-rate Kuramoto model and proved that the inertial terms do not affect synchronization conditions.

In view of the different biological and technological applications of the Kuramoto model \cite{RS-DP-NEL:07,KW-PC-SHS:98,FD-FB:09o-arxiv,PAT:03,DC-CPU:07,ARB-DJH:81,GBE:91,HDC-CCC:95}, similar tight and explicit bounds have to be derived for synchronization (as well as splay state stabilization) with arbitrary coupling topologies, phase and time delays, non-gradient-like dynamics, and possibly non-uniform coupling weights depending on state and time.

%%%%%%%%%%%%%%%%%%%%%%%%%%%%%%%%%%%%%%%%%%%%%%%%%%%%%%%%%%%%%%%%%%%%%%%%%

% -----------------------------------------------------------------------------------------------------%
%      Bibliography
% -----------------------------------------------------------------------------------------------------%

%\renewcommand{\baselinestretch}{0.96}
\bibliographystyle{siam}
\bibliography{alias,Main,FB,New}

% important: compile via LaTex -> LaTex -> BibTex -> LaTex !!!
%%%%%%%%%%%%%%%%%%%%%%%%%%%%%%%%%%%%%%%%%%%%%%%%%%%%%%%%%%%%%%%%%%%%%%%%%

\end{document}